\documentclass[a4paper,11pt,abstract=on]{scrartcl}

\usepackage{amsmath,amsthm,amsfonts,amssymb}
\usepackage{mathtools}

\newtheorem{theorem}{Theorem}[section]
\newtheorem{lemma}[theorem]{Lemma}
\newtheorem{corollary}[theorem]{Corollary}
\newtheorem{proposition}[theorem]{Proposition}
\theoremstyle{definition}
\newtheorem{remark}[theorem]{Remark}

\newtheorem{example}[theorem]{Example}

\usepackage[style=numeric,maxbibnames=99]{biblatex}
\usepackage{enumerate}
\usepackage{algorithm}
\usepackage[noend]{algpseudocode}
\usepackage{hyperref}
\usepackage{xcolor}
\usepackage{bbm}
\usepackage{pgfplots}

\usepackage{booktabs}
\usepackage{multirow}

\usepackage{nicefrac}
\usepackage{orcidlink}

\usepackage{comment}

\usepackage{wrapfig}

\allowdisplaybreaks

\usepackage{listings}
\usepackage{cleveref}

\definecolor{plotblue}{HTML}{0000FF}
\definecolor{plotorange}{HTML}{FFA500}
\definecolor{plotgreen}{HTML}{008000}
\definecolor{empiricalblue}{HTML}{1F77B4}
\definecolor{empiricalorange}{HTML}{D55E00}
\definecolor{empiricalgreen}{HTML}{009E73}
\definecolor{boundpurple}{HTML}{9467BD}

\newcommand{\legline}[2][]{%
  \tikz[baseline=-0.5ex]{%
    \draw[line width=0.8pt,#1] (0,0) -- (0.65,0);
    \if\relax\detokenize{#2}\relax\else
      \node[inner sep=0pt,#1] at (0.325,0) {$#2$};
    \fi
  }%
}

\newcommand{\legband}{%
  \tikz[baseline=-0.5ex]{%
    \fill[empiricalblue,opacity=0.2]
      (0,-0.08) rectangle (0.65,0.08);
    \draw[empiricalblue,line width=0.8pt]
      (0,0) -- (0.65,0);
  }%
}

\newcommand{\legbandred}{%
  \tikz[baseline=-0.5ex]{%
    \fill[red,opacity=0.2]
      (0,-0.08) rectangle (0.65,0.08);
  }%
}

\newcommand{\legbandunified}{%
  \tikz[baseline=-0.5ex]{%
    \fill[gray,opacity=0.2]
      (0,-0.08) rectangle (0.65,0.08);
    \draw[gray,line width=0.8pt]
      (0,0) -- (0.65,0);
  }%
}

\definecolor{codegreen}{rgb}{0,0.6,0}
\definecolor{codegray}{rgb}{0.5,0.5,0.5}
\definecolor{codepurple}{rgb}{0.58,0,0.82}
\definecolor{backcolour}{rgb}{0.95,0.95,0.92}

\lstdefinestyle{mystyle}{ 
    commentstyle=\color{codegreen},
    keywordstyle=\color{magenta},
    numberstyle=\tiny\color{codegray},
    stringstyle=\color{codepurple},
    basicstyle=*family\footnotesize,
    breakatwhitespace=false,         
    breaklines=true,                 
    captionpos=b,                    
    keepspaces=true,                 
    showspaces=false,                
    showstringspaces=false,
    showtabs=false,                  
    tabsize=2
}
\newcommand{\EE}{\mathbb{E}}
\newcommand{\PP}{\mathbb{P}}
\newcommand{\RR}{\mathbb{R}}

\newcommand{\NN}{\mathbb{N}}

\newcommand{\sphere}{\mathbb{S}}
\newcommand{\dd}{\mathrm{d}}

\newcommand{\Orth}{\mathrm{O}}

\DeclarePairedDelimiterX{\scp}[2]{\langle}{\rangle}{#1, #2}
\DeclarePairedDelimiter{\abs}{\vert}{\vert}
\DeclarePairedDelimiter{\norm}{\lVert}{\rVert}

\DeclareMathOperator{\sign}{sign}
\DeclareMathOperator{\diag}{diag}

\let\dim\relax
\DeclareMathOperator{\dim}{dim}
\DeclareMathOperator*{\argmax}{argmax}

\DeclareMathOperator*{\grad}{grad}
\DeclareMathOperator*{\spec}{spec}

\newcommand{\herA}{A^{\mathrm H}}

\usepackage{caption}
\title{Stochastic Rayleigh–Ritz Ascent: Global Linear Convergence under a Spectral Gap}

\author{Jonas Bresch\,\orcidlink{0009-0009-3855-0933}%
\thanks{Ludwig-Maximilians-Universität München, 
\href{https://www.math.lmu.de/math4ai/en/}{Bavarian AI Chair for Mathematical Foundation of Artificial Intelligence}, Akademiestr. 7, 80799 Munich, Germany. \href{mailto:jonas.bresch@math.lmu.de}{jonas.bresch@math.lmu.de}.}
\thanks{Munich Center for Machine Learning (MCML).}
\and
Oleh Melnyk,\orcidlink{0000-0002-1816-7613}%
\thanks{Institute of Mathematics, Technical University Berlin, Straße des 17. Juni 136, 10623 Berlin, Germany,
\href{mailto:melnyk@math.tu-berlin.de}{melnyk@math.tu-berlin.de}. 
\url{https://www.tu.berlin/imageanalysis}.}
}

\begin{document}
\maketitle

\begin{abstract}
    We consider the adjoint-free random-plane iteration algorithm for finding the top eigenvector of a matrix proposed in \cite{bresch2024matrixfreestochasticcalculationoperator}. Unlike standard methods such as the power method or gradient descent, the above algorithm lacked linear convergence guarantees in the presence of a spectral gap. In this paper, we establish explicit spectral-gap-dependent global linear convergence rates in expectation, providing quantitative guarantees for the empirically observed convergence. 
\end{abstract}

\noindent\textbf{AMS Classification:}
65F35, 
15A60, 
68W20  

\noindent\textbf{Keywords:} 
    Operator norm,
    generalized Rayleigh quotient,
    stochastic Rayleigh–Ritz,
    zeroth-order method,
    convergence analysis,
    linear rate

\section{Introduction}
\label{sec:intro}

Estimating the operator norm of a matrix $A$ 
amounts to maximizing the Rayleigh quotient of $A^*A$ 
over the unit sphere
\begin{equation}
    \label{eq: prob}
    \|A\|^2
    = \max_{\|v\|=1} \langle v,A^*Av\rangle
    = \max_{\|v\|=1} \|Av\|^2.
\end{equation}
Although the objective can be evaluated using only forward
applications of $A$, classical algorithms for finding the maximizer, such as the power method \cite{kuczynski1992estimating,delcorso1997estimating,parlett1998symmetriceigenvalueproblem} or Riemannian gradient ascent \cite{alimisis2021distributed}, require access to the adjoint $A^*$.
An exact adjoint implementation, however, may be unavailable or computationally expensive in certain scenarios. For instance, in imaging applications such as computed tomography~\cite{buzug2008ct}, the adjoint is often replaced by a computationally tractable approximation and discretization scheme
to reduce memory use and improve computational speed on GPUs \cite{van2016fast,hansen2021GMRESMF,huber2025convergence,hansen2021CT}.
The resulting adjoint mismatch can adversely affect reconstruction quality%
~\cite{xie2015effective,zhang2016unmatchedprojback,peterson2017monte,bresch2024matrixfreestochasticcalculationoperator,bresch2025computingadjointmismatchlinear}.
For operator-norm estimation, replacing $A^*$ by an approximation or a discretization, say $V^* \approx A^*$, also changes the operator used in the iteration, cf.~\cite[Sec.~1.1]{bresch2024matrixfreestochasticcalculationoperator} and~\cite[Sec.~3.3]{bresch2025computingadjointmismatchlinear}, and therefore may be misleading for a successful application of power-schemes \cite[Sec.~28~\&~48]{trefethen2005spectrapseudospectra}.
More precisely, if $V^*A$ is non-normal, power-schemes may fail to converge or may return estimates far from the leading singular value. A full understanding of the impact of $V^*$ requires further analysis \cite[p.~446]{trefethen2005spectrapseudospectra}, which should be bypassed by avoiding the use of either $V^*$ or $A^*$.
These considerations motivate adjoint-free methods that estimate the norm of $A$ using only forward evaluations so that they can be applied to black-box implementations while retaining provable convergence guarantees.

Random-plane optimization provides a simple approach for solving \eqref{eq: prob} without involving the adjoint $A^*$.
In \cite{bresch2024matrixfreestochasticcalculationoperator}, the authors proposed a method that at each iteration samples a direction in a tangent plane to the sphere at the current iterate. Then, the objective is exactly maximized over the unit circle in the plane spanned by the current iterate and the sampled direction. The resulting subproblem can be solved using only applications of $A$ and inner products. The detailed description of the algorithm can be found in Sec.~\ref{sec: preliminaries}.
This approach was generalized to the estimation of adjoint mismatch between two linear operators in~\cite{bresch2025computingadjointmismatchlinear} and extended to estimate generalized Rayleigh quotients in~\cite{bresch2025stochastic}. Together, these works establish sublinear convergence guarantees and residual estimates, while the numerical trials therein indicate much faster linear convergence rates.
Results in~\cite{bresch2025stochastic} show that the random-plane iteration can be viewed as a Riemannian zeroth-order method with an optimized step size, and propose a variance reduction through multi-sample gradient estimators. Given this connection to zeroth-order optimization methods, we explore the connections to other algorithms in the literature, indicate their linear convergence guarantees, and discuss their differences.

Riemannian zeroth-order gradient methods \cite{li2023stochastic} with constant step size are shown to converge sublinearly to the critical points of non-convex functions. Their Euclidean counterparts, e.g., in \cite{Ji2019ImprovedZeroth-Order, kazemi2024efficient}, exhibit stronger (local) linear rates under the (local) Polyak--\L{}ojasiewicz (PL) condition. In the context of maximizing the nonconstant Rayleigh quotient \eqref{eq: prob}, nonleading eigenvectors are critical points that are not globally optimal, limiting the performance of the mentioned general guarantees.
Inexact Riemannian gradient descent~\cite{zhou2025inexact} provides another possible
framework through controlled gradient errors and Kurdyka--\L{}ojasiewicz (KL) inequalities. A single random tangent direction does not, in general, satisfy the uniform relative-accuracy requirement of that analysis, while its alternative absolute-error conditions would require separate verification.

In contrast, first-order methods analyzed specifically for eigenvalue problem \eqref{eq: prob} provide more insight into the local geometry.
Riemannian steepest ascent~\cite{alimisis2021distributed, alimisis2024geodesic,alimisis2024gradient} and the randomized submanifold method~\cite{Han2025EfficientOptimizationOrthogonalityConstraint} converge linearly under the Riemannian PL condition. In particular, in~\cite{alimisis2021distributed} the rate of convergence is directly linked to the spectral gap of $A^*A$. 
However, these methods use the adjoint to perform updates and are not available in the setting of this paper.

Motivated by these observations, we establish explicit local and global linear rates for the expected objective error under a spectral-gap assumption. Our local rate estimate, which depends on the spectral gap, spectral width, and dimension, is near-optimal up to a constant multiplicative factor that corresponds to the size of the convergence neighborhood.

The paper is structured in the following way.
The preliminaries in Sec.~\ref{sec: preliminaries} cover necessary background on the random-plane algorithm for estimating the operator norm \cite{bresch2024matrixfreestochasticcalculationoperator}.
Sec.~\ref{sec: contributions} states the main result.
Sec.~\ref{sec: qualitative ascent bounds} establishes the spectral inequalities and their connection to local expected linear convergence.
We discuss results of Sec.~\ref{sec: qualitative ascent bounds} in the context of the PL condition in Sec.~\ref{sec: PL KL background}. 
Sec.~\ref{sec: linear convergence rate} derives the global linear convergence in expectation.
In Sec.~\ref{sec: generalized quotient}, we extend the analysis to the zeroth-order estimation for the generalized Rayleigh quotient considered in \cite{bresch2025stochastic}. We compare our theoretical findings with the empirical performance of the random-plane algorithm in Sec.~\ref{sec:num}, followed by a conclusion.

\subsection{Preliminaries}
\label{sec: preliminaries}

\paragraph{Algorithm.} 
We consider the problem of finding the eigenvector corresponding to the leading eigenvalue of $A^*A$ for a given matrix $A \in \RR^{m \times d}$ without involving the adjoint $A^*$. This can be achieved through the stochastic random-plane algorithm \cite[Alg.~1]{bresch2024matrixfreestochasticcalculationoperator}. For our analysis, we focus on its Rayleigh-Ritz reformulation discussed in \cite[Sec.~8.6]{bresch2025stochastic}. That is, let $v^0 \in \sphere^{d-1} \coloneqq \{ x \in \RR^{d} \mid \norm{x} = 1 \}$ be the initial guess sampled uniformly at random, with $\norm{\cdot}$ denoting the Euclidean norm in $\RR^d$.  
Then, the sequence of iterates $(v^k)_{k \ge 0}$ is constructed by iteratively solving 
\begin{equation}
    \label{eq: plane}
    v^{k+1}\in\argmax_{\substack{
    z\in\operatorname{span}\{v^k,x^k\}\\ \|z\|=1}}\|Az\|^2,
\end{equation}
where the direction $x^k$ is sampled uniformly at random on the unit sphere in the tangent space $T_{v^k} = \{x \in \RR^d \mid \scp{x}{v^k} = 0\}$ to the sphere $\sphere^{d-1}$ at the point $v^k$. This is achieved by sampling $y \sim \mathcal N(0,I_d)$, the standard Gaussian random vector, and setting
\begin{equation}
    \label{eq: construction}
     x^k \coloneqq \frac{P_{T_{v^k}} y}{\norm{P_{T_{v^k}} y}},
\end{equation}
where $P_{T_{v^k}} \coloneqq I_d - v^k(v^k)^*$ is the orthogonal projector
onto the tangent space and $I_d$ is the $d \times d$ identity matrix.

\paragraph{Probability Space and Distributions.}

We can associate the samples $(v^k)_{k \ge 0}$ and $(x^k)_{k \ge 0}$ to the sequences of random variables.
Following \cite[Rem.~4.1]{bresch2025stochastic}, we consider a probability space $(\Omega, \mathcal A, \PP)$
and a sequence of random variables $V^k : \Omega \to \sphere^{d-1}$ with law \smash{$\PP_{V^k} \coloneqq V^k_{\#} \PP = \PP (V^k)^{-1}$}. By construction, $\PP_{V^0} = \sigma_{\sphere^{d-1}}$, the Haar measure on the $(d-1)$-dimensional unit sphere and $v^0 \sim V^0$. Moreover, for each sample $v^k$, we consider random variables $X^k : \Omega \to T_{v^k} \cap \sphere^{d-1} \simeq \sphere^{d-2}$ almost surely. Since they depend on $v^k$, by \cite[Lem.~4.1]{bresch2025stochastic}, their conditional distribution given $V^k = v^k$ is uniform, i.e., $\PP_{X^k \mid V^k = v^k} = \mathcal U(T_{v^k} \cap \sphere^{d-1}) = \sigma_{\sphere^{d-2}}$.
As a result, we can view $(V^k)_{k \in \mathbb N}$ as a Markov chain with values on $\sphere^{d-1}$. Furthermore, we say that something holds almost surely (a.s.) if its exceptional event has probability zero.

\paragraph{Solutions of \eqref{eq: plane}.}

Since $v^k$ and $x^k$ are orthogonal by construction, any $z \in \operatorname{span}\{v^k,x^k\}$ with $\norm{z}=1$ can be written as $z = u_1 v^k + u_2 x^k$ with $\norm{u} = 1$. The subproblem \eqref{eq: plane} then corresponds to finding the leading eigenvalue $\rho_{k+1}$ and eigenvector $u \in \RR^2$ of the system matrix
\begin{equation}\label{eq: coefs}
    H_k \coloneqq \begin{pmatrix} \scp{A v^k}{A v^k} & \scp{A v^k}{A x^k} \\ \scp{A v^k}{A x^k} & \scp{A x^k}{A x^k} \end{pmatrix} \eqqcolon \begin{pmatrix}\rho_k&a_k\\a_k&s_k\end{pmatrix},
\end{equation}
see \cite[Rem.~3.4]{bresch2026genoperatornorm} for further details. 

If $a_k = 0$, $H_k$ is a diagonal matrix and the maximizer is determined by the larger of the values $\rho_k$ and $s_k$,
\[
\rho_{k+1} = \begin{cases}
    \rho_k, & \rho_k \ge s_k, \\
    s_k, & \rho_k < s_k,
\end{cases}
\qquad
v^{k+1} = \begin{cases}
    v^k, & \rho_k \ge s_k, \\
    x^k, & \rho_k < s_k.
\end{cases}
\]
Otherwise, if $a_k \neq 0$, according to \cite[Prop.~2.4]{bresch2024matrixfreestochasticcalculationoperator}, the unique maximizer of \eqref{eq: plane} is given by
\begin{equation*}
    \label{eq: update}
    v^{k+1} = \frac{v^k + \tau^k x^k}{\norm{v^k + \tau^k x^k}}, 
    \quad 
    \tau^k \coloneqq \frac{s_k - \rho_k}{2a_k} + \sign(a_k) \sqrt{\tfrac{(s_k - \rho_k)^2}{4a_k^2} + 1}.
\end{equation*}
In all cases, we have 
\begin{equation}
    \label{eq: nondecreasing}
    \norm{A v^{k+1}}^2 = \rho_{k+1} \geq \rho_k = \norm{A v^{k}}^2,
\end{equation}
and the algorithm produces a nondecreasing sequence $(\rho_k)_{k \ge 0}$. 

\paragraph{Sublinear Convergence.}

To put the generated sequence $(v^k)_{k \ge 0}$ in perspective of finding the leading eigenvector of $A^*A$ for $A \in \RR^{m \times d}$, let us denote the eigenvalues of $A^*A \in \RR^{d\times d}$ by $\lambda_1 \ge \lambda_2\geq\cdots\geq\lambda_d\geq0$. Throughout the paper, the eigenvectors $u$ are always unit vectors, i.e., $\norm{u}=1$. 
We define the spectral error as
\begin{equation}
    \label{eq: eps v}
     \varepsilon_k 
     \coloneqq \varepsilon(v^k)
     \coloneqq \lambda_1 - \rho_k 
     = \lambda_1 - \norm{A v^k}^2
     = \max_{\norm{v}=1} \norm{A v}^2 - \norm{A v^k}^2
\end{equation}
and the residual at $k$th iteration as
\begin{equation}
    \label{eq: res}
     r^k 
     \coloneqq 
     r(v^k)
     \coloneqq A^*Av^k - \rho_k v^k.
\end{equation}
The main results of \cite{bresch2025stochastic} provide the following convergence guarantees for $\varepsilon_k$ and $r^k$.

\begin{theorem}[{\cite[Thm.~4.5,~5.6 \& Lem.~5.7 \& Thm.~6.3]{bresch2025stochastic}}] 
    \label{thm: sublinear}
    Let 
    \begin{equation}\label{eq: dimension}
        d_{\max} \coloneqq \dim(\{v \in \RR^d : A^*A v = \lambda_1 v\})
    \end{equation}
    be the dimension of the eigenspace corresponding to the leading eigenvalue. Consider the iterates $(v^k)_{k\ge 0}$ generated by iteration \eqref{eq: plane} with $v^0$ sampled uniformly at random on $\sphere^{d-1}$. If $d_{\max} = d$ or $d_{\max} = d-1$, then $v^0$ or, respectively, $v^1$ is the leading eigenvector of $A^* A$ a.s. 
    Otherwise, if $d_{\max} < d-1$, we have  
    \[
    \varepsilon_k \to 0, \quad \text{as} \ k \to \infty, \ \text{a.s.}
    \quad \text{and} \quad 
    \min_{k = 0, \ldots, K}\EE[\norm{r^k}^2] \le \frac{32 (d-1) \norm{A}^4}{K+1}.
    \]
    Furthermore, if $d_{\max} = 1$, then $v^k \to u_1$ as $k \to \infty$ a.s., where $u_1$ is the unique (up to a sign) unit eigenvector of $A^* A$ corresponding to $\lambda_1$. 
\end{theorem}

\subsection{Main Contributions}
\label{sec: contributions}

While Thm.~\ref{thm: sublinear} establishes almost-sure convergence of the spectral error $\varepsilon_k$ and the iterates \eqref{eq: plane} and a sublinear bound for the expected squared residual, it does not quantify the decay of $\mathbb E[\varepsilon_k]$. 
In contrast, other algorithms for finding the leading eigenvector, such as the (randomized) power method \cite{vonMises1929power,kuczynski1992estimating,halko2011randomapproximatematrixdecomposition} or Riemannian gradient ascent \cite{alimisis2024geodesic,XuKeGao2017}, converge at a linear rate when the spectral gap is positive. 
Our main contribution lies in closing this gap and establishing global linear convergence guarantees in expectation under a positive spectral gap, summarized in the following main theorem.

\begin{theorem}[Main Theorem]
    \label{thm: contribution}
    Let $d\ge 3$ and suppose that the leading
    eigenvalue of $A^*A$ is simple with corresponding unit eigenvector $u_1$, 
    i.e., $d_{\max} = 1$ from \eqref{eq: dimension}, 
    with positive spectral gap and spectral width
    \begin{equation}
        \label{eq: constants}
        \delta \coloneqq \lambda_1-\lambda_2>0,
        \qquad
        \Delta \coloneqq \lambda_1-\lambda_d \geq \delta > 0.
    \end{equation}
    For the stochastic random-plane iteration \eqref{eq: plane}, we establish linear convergence rates
    \[
        \mathbb E[\varepsilon_k]
        \le C \varrho^k,
        \quad
        \EE\bigl[\|r^k\|^2\bigr] \leq C \Delta \varrho^k
        ,
        \quad \text{and} \quad
        \EE \biggl[\min_{s\in\{-1,1\}}\|v^k-su_1\|^2\biggr] \leq \frac{2C}{\delta} \varrho^k,
        \quad k\ge 0, 
    \]
    where $C>0$ and $\varrho\in(0,1)$ are explicit constants
    depending only on $d$, $\delta$, and $\Delta$
    given in \eqref{eq: q}, \eqref{eq: p} 
    and Thm.~\ref{thm: global}, respectively Thm.~\ref{thm: global-unequal}. 
\end{theorem}

Thus, under the spectral-gap assumption, Thm.~\ref{thm: contribution} provides a quantitative linear rate for spectral optimality as well as stronger residual convergence guarantees.

The proof combines two estimates:
First, for any $\eta\in(0,1)$, 
we establish local expected contraction 
\[
    \mathbb E_{V^{k+1} \mid V^k}[\varepsilon_{k+1}]
    \le
    \left(1-\frac{(1-\eta)\delta}{(d-1)\Delta}\right)
    \varepsilon_k
\]
in the neighborhood of global maximum described by the set $\{\varepsilon_k\le\eta\delta\}$ with $0< \eta < 1$.
Second, we derive a geometric tail bound for the first
entrance time into this region. Combining these estimates
yields the global rate from iteration zero,
see Thm.~\ref{thm: local},~\ref{thm: global}, and~\ref{thm: global-unequal}. 
Then, we derive convergence rates for the squared residual norm $r^k$ and the squared distance to the leading eigenspace by bounding these values in terms of $\varepsilon_k$ in Cor.~\ref{eq: other-errors}.

\section{Linear Convergence under a Spectral Gap}
\label{sec:linear_convergence}

Throughout the remainder of the paper, we rely on assumptions of Thm.~\ref{thm: contribution}. The case $d \le 2$ is excluded, since by \Cref{thm: sublinear}, the algorithm will terminate after at most one iteration a.s.

Since we are interested in a linear rate of 
\begin{equation*}
    0\leq \varepsilon_{k+1} \leq \varepsilon_k\leq \Delta
    \quad \forall k \in \mathbb N
    \qquad \text{a.s.}
\end{equation*}    
by \eqref{eq: nondecreasing} and \eqref{eq: eps v}, we first consider the local linear convergence in the neighborhood of the solution in \Cref{sec: qualitative ascent bounds} and then study the probability of entering it in \Cref{sec: linear convergence rate}.

\subsection{Ascent Bounds and Local Contraction}
\label{sec: qualitative ascent bounds}

In this subsection we first derive qualitative bounds on the expected error. We start with an analogue of sufficient decrease condition of $\varepsilon_k$ in terms of $\norm{r^k}^2$.

\begin{lemma}
    \label{lem: ascent}
    For every $k\geq0$ it holds 
    \begin{equation*}
        \label{eq: expected-ascent}
        \EE_{x^{k} \sim X^{k}\mid V^k = v^k}[\varepsilon_{k+1}]
        \leq \varepsilon_k-\frac{\|r^k\|^2}{(d-1)\Delta}.
    \end{equation*}
\end{lemma}
\begin{proof}
    Utilizing \eqref{eq: coefs}
    and characteristic equation 
    gives
    \begin{align*}
        0 = \det(H_k - \rho_{k+1} I_2)
        &= (\rho_{k+1}-\rho_k)(\rho_{k+1}-s_k) - a_k^2 
        \; 
        \Leftrightarrow
        \; 
        a_k^2
        = (\rho_{k+1}-\rho_k)(\rho_{k+1}-s_k).
    \end{align*}
    Both $\rho_{k+1}$ and $s_k$ belong to $[\lambda_d,\lambda_1]$,
    and $\rho_{k+1}\geq\max\{\rho_k,s_k\}$. 
    Consequently, by \eqref{eq: eps v}, we get
    \begin{equation}
        \label{eq: pathwise-ascent}
        a_k^2
        \leq \Delta (\rho_{k+1}-\rho_k)
        = \Delta (\varepsilon_k-\varepsilon_{k+1})
        \quad \Leftrightarrow \quad 
        \varepsilon_{k+1} \leq \varepsilon_k - \frac{a_{k}^2}{\Delta}.
    \end{equation}
    Now, we transform $a_k^2$. By \eqref{eq: coefs}, we have
    \begin{align}
        \label{eq: ak}
        a_k 
        = \scp{A^*A v^k}{x^k} 
        = \scp[\Bigl]{\underbrace{A^*A v^k - \rho_k v^k}_{= r^k}}{x^k} + \rho_k \underbrace{\scp{v^k}{x^k}}_{= 0}
        = \scp{r^k}{x^k},
    \end{align}
    since $x^k \perp v^k$ by construction in \eqref{eq: construction}. Next, \cite[Lem.~4.1]{bresch2025stochastic} gives
    \begin{equation}
        \label{eq: E xx}
        \EE_{x^k \sim X^k \mid V^k = v^k}[x^k(x^k)^*] 
        = \tfrac{1}{d-1}[I_d-v^k(v^k)^*],
    \end{equation}
    which yields
    \begin{align}
        \label{eq: exp ak}
        \EE_{x^k \sim X^k \mid V^k = v^k}[a_k^2]
        &\overset{\eqref{eq: ak}}{=} \EE_{x^k \sim X^k \mid V^k = v^k}[\scp{r^k}{x^k}^2] 
        = \EE_{x^k \sim X^k \mid V^k = v^k}[(r^k)^*x^k (x^k)^* r^k] \notag\\
        &= (r^k)^* \, \EE_{x^k \sim X^k \mid V^k = v^k}[x^k (x^k)^*] \, r^k 
        = \tfrac{1}{d-1} (r^k)^* [I_d - v^k (v^k)^*] r^k \notag\\
        &= \tfrac{1}{d-1} (\|r^k\|^2 - \scp{v^k}{r^k}^2).
    \end{align}
    Finally, we notice that \eqref{eq: res} implies
    \begin{equation}
        \label{eq: r orth}
        \scp{r^k}{v^k}
        = \scp{A^*A v^k - \rho_k v^k}{v^k}
        = \underbracket{\norm{A v^k}^2}_{= \rho_k} - \rho_k \underbracket{\norm{v^k}^2}_{= 1}
        = 0.
    \end{equation}
    Consequently, taking the conditional expectation in  \eqref{eq: pathwise-ascent} and combining it with \eqref{eq: exp ak} and \eqref{eq: r orth} proves the assertion.
\end{proof}

Next, we relate $\norm{r^k}^2$ to the errors $\varepsilon_k$.

\begin{lemma}
    \label{lem: residual}
    For fixed $v^k \in \sphere^{d-1}$,
    it holds
    \begin{equation*}
        \label{eq: residual-bound}
        \varepsilon_k(\delta-\varepsilon_k) 
        \le \norm{r^k}^2
        = \|A^*A v^k - \rho_k v^k\|^2 
        \le \varepsilon_k(\Delta - \varepsilon_k)
        \leq \Delta \varepsilon_k.
    \end{equation*}
\end{lemma}
\begin{proof}
    Let $(u_i)_{i=1}^d$ 
    be an orthonormal eigenbasis of $A^*A$ 
    and write $v^k = \sum_{i=1}^d \gamma_i^k u_i$
    for some $\{\gamma_i^k\}_{i = 1}^d \subset \RR$. 
    Then 
    \begin{equation}
        \label{eq: delta gamma}
        1 = \norm{v^k}^2
        = \sum_{i = 1}^d (\gamma^k_i)^2 \norm{u_i}^2
        = \sum_{i = 1}^d (\gamma^k_i)^2.
    \end{equation}
    Set $\delta_i \coloneqq \lambda_1-\lambda_i$, so that $\delta_1 = 0$, 
    $\delta_2 = \delta$,
    and $\delta_i\geq\delta$ for $i\geq2$.
    We have $\varepsilon_k$ as 
    \begin{align}
        \label{eq: eps k}
        \varepsilon_k 
        &= \lambda_1 - (v^k)^* A^*A v^k 
        = \lambda_1 - (v^k)^*\sum_{i = 1}^d \lambda_i \gamma_i^k u_i 
        = \lambda_1 - \sum_{i,j = 1}^d \gamma_j^k \lambda_i \gamma_i^k \underbracket{(u_j)^* u_i}_{= \delta_{ij}} 
        \notag 
        \\[-11pt]
        &= \lambda_1 - \sum_{i = 1}^d (\gamma_i^k)^2 \lambda_i 
        \overset{\eqref{eq: delta gamma}}{=} \sum_{i = 1}^d (\gamma_i^k)^2 (\lambda_1 - \lambda_i)
        = \sum_{i = 1}^d \delta_i (\gamma_i^k)^2.
    \end{align} 
    On the other hand, we have
    \[
    r^k = A^*A v^k - \rho_k v^k = \sum_{i = 1}^d (\lambda_i \gamma_i^k - \rho_k \gamma_i^k) u_i,
    \]
    using that $\{u_i\}_{i = 1}^d$ is a basis. Therefore,
    \begin{align}
        \label{eq: rk}
        \norm{r^k}^2
        &= \sum_{i = 1}^d (\lambda_i - \rho_k)^2 (\gamma_i^k)^2 
        = \sum_{i = 1}^d (\lambda_1 - \rho_k - (\lambda_1 - \lambda_i))^2 (\gamma_i^k)^2 
        = \sum_{i = 1}^d (\varepsilon_k - \delta_i)^2 (\gamma_i^k)^2 
        \notag
        \\[-3pt]
        &= \sum_{i = 1}^d \delta_i^2 (\gamma_i^k)^2
        - 2 \varepsilon_k \sum_{i = 1}^d \delta_i (\gamma_i^k)^2 
        + \varepsilon_k^2 \sum_{i = 1}^d (\gamma_i^k)^2 
        \overset{\eqref{eq: delta gamma}, \eqref{eq: eps k}}{=} \sum_{i = 1}^d \delta_i^2 (\gamma_i^k)^2 - \varepsilon_k^2.
    \end{align}
    Finally,
    by using $\delta_1^2 = 0 = \delta \delta_1$ and $\delta_i \geq \delta > 0$ for $i\geq 2$, we obtain 
    $\delta_i^2 \geq \delta_i\delta \geq 0$, 
    and the assertion follows by bounding
    \begin{align*}
        \norm{r^k}^2
        &\overset{\eqref{eq: rk}}{=} \sum_{i=1}^d \delta_i^2 (\gamma_i^k)^2 - \varepsilon_k^2
        \geq \delta \sum_{i=1}^d \delta_i (\gamma^k_i)^2 - \varepsilon_k^2
        \overset{\eqref{eq: eps k}}{=} \varepsilon_k(\delta-\varepsilon_k).
    \end{align*}
    The right hand side, follows by using $\delta_i\leq \Delta$,
    \[
        \norm{r^k}^2
        \overset{\eqref{eq: rk}}{=} \sum_i \delta_i^2 (\gamma_i^k)^2-\varepsilon_k^2
        \leq \Delta \sum_{i = 1}^d \delta_i (\gamma_i^k)^2 - \varepsilon_k^2
        \overset{\eqref{eq: eps k}}{=} \varepsilon_k(\Delta - \varepsilon_k).
        \qedhere
    \]
\end{proof}

In the following we consider neighborhood of the global maximum, where errors satisfies $\varepsilon_k \leq \eta \delta$ for some $\eta \in (0,1)$. This condition is defined by the set 
\[
    v^k 
    \in 
    \mathcal E_{\eta,\delta} \coloneqq \{w \in \sphere^{d-1} 
            \mid \varepsilon(w) = \lambda_1 - \norm{A w}^2 \leq \eta \delta\}.
\]
Next, we show that within this set we obtain a linear convergence rate.

\begin{theorem}
    \label{thm: local}
    Let $\eta\in(0,1)$ 
    and define
    \begin{equation}
        \label{eq: q}
        q = q(\eta, \delta, \Delta) 
        \coloneqq 1-\frac{(1-\eta)\delta}{(d-1)\Delta}
        \in (0,1).
    \end{equation}
    Then, we have
    \begin{equation}
        \label{eq: contraction}
        \EE_{x^{k} \sim X^{k} \mid V^k = v^k}[\varepsilon_{k+1}]
        \leq  
        q \varepsilon_k,
        \quad \text{for} \quad 
        v^k \in \mathcal E_{\eta,\delta}.
    \end{equation}
    In particular, if $v^0 \in \mathcal E_{\eta,\delta}$ a.s.,
    then
    \begin{equation}
        \label{eq: local-rate}
        \EE[\varepsilon_{k}]
        \leq 
        q^k \EE[\varepsilon_0],
        \qquad\text{for every }k\geq0.
    \end{equation}
\end{theorem}
\begin{proof}
    By Lem.~\ref{lem: ascent} and~\ref{lem: residual} 
    we obtain
    \[
        \EE_{x^{k} \sim X^{k} \mid V^k = v^k}[\varepsilon_{k+1}]
        \leq \varepsilon_k-\frac{\varepsilon_k(\delta-\varepsilon_k)}{(d-1)\Delta}.
    \]
    For $v^k \in \mathcal E_{\eta,\delta}$ we have
    $\delta-\varepsilon_k\geq(1-\eta)\delta$, 
    which proves~\eqref{eq: contraction}.
    Since $(\varepsilon_k)_{k\geq0}$ is nonincreasing, this event also holds for subsequent iterations. Equation \eqref{eq: local-rate} follows by induction and the tower property,
    \[
        \EE[\varepsilon_{k+1}]
        = \EE[ \EE_{x^{k} \sim X^{k} \mid V^k = v^k}[\varepsilon_{k+1}] ]
        \leq 
        q \EE[\varepsilon_{k}] \le \ldots \le q^{k+1} \EE[\varepsilon_0].
        \qedhere
    \]
\end{proof}

Choosing $\eta=\tfrac12$ gives the conditional contraction factor
\[
    q = q(\nicefrac12, \delta,\Delta)
    = 1-\frac{\lambda_1-\lambda_2}{2(d-1)(\lambda_1-\lambda_d)}
\]
once the iterates $v^k$ satisfy $\|Av^k\|^2\geq(\lambda_1+\lambda_2)/2$.
The relevant gap is therefore the gap of $A^*A$, normalized by its spectral width. 
Since $\Delta\leq\lambda_1$, a simpler, possibly weaker factor is 
\[
    1-\frac{\lambda_1-\lambda_2}{2(d-1)\lambda_1}
    \geq 
    q(\nicefrac12,\delta,\Delta).
\]
These are upper bounds on the expected error, and do not assert a deterministic contraction at each realization of a random step. Moreover, the obtained rate is tight up to the multiplier $1-\eta$, as indicated by the next example.

\begin{example}\label{example}
Let $u_1 \in \sphere^{d-1}$, $0 < \delta < \lambda_1$ be arbitrary.   
Consider a matrix $A$ given by
\begin{equation*}
    A = \sqrt{\lambda_1 - \delta} I + (\sqrt{\lambda_1} - \sqrt{\lambda_1-\delta})u_1 u_1^*,
    \quad 
    \text{with}
    \quad 
    A^* A = (\lambda_1 - \delta) I+ \delta u_1u_1^*,
\end{equation*}
so that $\delta = \Delta$ and $\lambda_1$. Let $\alpha_k=\langle u_1,v^k\rangle,\beta_k=\langle u_1,x^k\rangle.$
Then, for the system matrix \eqref{eq: coefs} in $k$th iteration is given by 
\[
    H_k
    = \begin{pmatrix} v^k & x^k \end{pmatrix}^* A^* A \begin{pmatrix} v^k & x^k \end{pmatrix}
    = (\lambda_1-\delta)I_2
    +\delta
    \begin{pmatrix}\alpha_k\\ \beta_k\end{pmatrix}
    \begin{pmatrix}\alpha_k&\beta_k\end{pmatrix}
\]
with 
\[
    (H_k)_{1,1} = \rho_k = \langle v^k, A^*A v^k\rangle = \lambda_1 - \delta + \delta |\alpha_k|^2.
\]
Its leading eigenvalue is $\rho_{k+1}=\lambda_1-\delta+\delta(\alpha_k^2+\beta_k^2)$.
Hence the exact ascent is $\varepsilon_k-\varepsilon_{k+1} = \rho_{k+1} - \rho_k = \delta\beta_k^2.$
By utilizing \eqref{eq: E xx}, we obtain 
\begin{align*}
    \EE_{x^k \sim X^k \mid V^k = v^k}\bigl[\beta_k^2\bigr]
    & = \EE_{x^k \sim X^k \mid V^k = v^k}\bigl[ u_1^* x^k (x^k)^* u_1 \bigr]
    = \frac{1}{d-1} u_1^* [I_d - v^k (v^k)^*] u_1 \\
    & = \frac{1-\alpha_k^2}{d-1} 
    = \frac{\delta  - \delta \alpha_k^2 + \lambda_1 - \lambda_1}{\delta(d-1)}
    = \frac{\lambda_1 - \rho_k}{\delta(d-1)}
    = \frac{\varepsilon_k}{\Delta(d-1)}.
\end{align*}
and, consequently, by iterating we get
\begin{equation*}
    \label{eq: rd}
    \EE[\varepsilon_k]
    = \Bigl(1-\frac{\delta}{\Delta(d-1)}\Bigr)^k \EE[\varepsilon_0].
\end{equation*}
Compared with Thm.~\ref{thm: local}, this rate matches the rate without the $1-\eta$ multiplier, i.e. using $\eta = 0$. Therefore, our local linear convergence rate is near-optimal and we refer to the optimal local linear convergence rate as $q_0 \coloneqq q(0,\delta,\Delta) = 1 - \frac{\delta}{\Delta(d-1)}$.  
\end{example}

The dimension dependence of our convergence guarantee contrasts with that of the randomized power method applied to $A^*A$, where the initialization is uniformly sampled on the unit sphere \cite{kuczynski1992estimating,delcorso1997estimating}. Under the assumption of a positive spectral gap, the contraction factor becomes $\lambda_2/\lambda_1 = 1 - \delta/\lambda_1$, without an extra $d-1$ factor in the denominator. In view of Ex.~\ref{example}, the random plane algorithm encounters the curse of dimensionality. Whether a different adjoint-free algorithm may avoid it is an open problem.

\subsection{Interpretation through the PL Condition}
\label{sec: PL KL background}

Thm.~\ref{thm: local} admits an interpretation through the
Polyak--\L{}ojasiewicz (PL)
condition~\cite{Polyak1963,Karimi2016LinearConvergence}.
A differentiable function $H$ on a Riemannian manifold
satisfies a PL inequality on $\mathcal U$, 
relative to a minimum value $H_*$, 
if
\[
    \|\grad H(v)\|^2
    \geq 2\mu\bigl(H(v)-H_*\bigr),
    \qquad v\in\mathcal U,\quad \mu>0,
\]
where $H(v)\geq H_*$ on $\mathcal U$ 
and the gradient norm is induced by the Riemannian metric; 
see also~\cite[Definition~5.4]{Han2025EfficientOptimizationOrthogonalityConstraint}.
This gradient-domination condition is implied 
by Euclidean strong convexity, 
but requires neither convexity nor a unique minimizer~\cite{Karimi2016LinearConvergence}.

For $\varepsilon(v) \coloneqq \lambda_1-\|Av\|^2$ on $\sphere^{d-1}$ as in \eqref{eq: eps v},
we have $\min_{\norm{v}=1} \varepsilon(v) = 0$ and $\grad \varepsilon(v) = -2r(v)$,
where $r(v) \coloneqq A^*Av-\|Av\|^2v$ is defined in \eqref{eq: res}.
Lem.~\ref{lem: residual} therefore gives, for $\eta\in(0,1)$
and $v\in\mathcal E_{\eta,\delta}$,
\[
    \|\grad \varepsilon(v)\|^2
    \geq 4\varepsilon(v)\bigl(\delta-\varepsilon(v)\bigr)
    \geq 4(1-\eta)\delta \varepsilon(v).
\]
Thus, $\varepsilon$ satisfies a PL inequality on $\mathcal E_{\eta,\delta}$ 
with $\mu=\mu_{\eta,\delta}=2(1-\eta)\delta$, 
without requiring geodesic strong convexity. 
Combining this inequality with Lem.~\ref{lem: ascent} yields
\[
    \EE_{X^k\mid V^k=v}\left[\varepsilon(V^{k+1})\right]
    \leq \varepsilon(v)
    -\frac{\|\grad \varepsilon(v)\|^2}{4(d-1)\Delta}
    \leq q
    \varepsilon(v),
    \qquad v\in\mathcal E_{\eta,\delta},
\]
where $q$ is defined in~\eqref{eq: q}. 

Gradient domination was also used to establish local linear convergence rates for Riemannian gradient ascent for finding the leading eigenvector in \cite[Prop. 4]{alimisis2021distributed} under the assumption of the spectral gap and with the exact logarithmic map on the sphere. 

However, a global PL inequality relative to $\min_{\norm{v}=1}\varepsilon(v) = 0$ cannot hold since it implies that the function only has global minima. This condition fails since nonleading eigenvectors are critical points with $\varepsilon(v) > 0$.
Moreover, the local PL inequality does not quantify the time needed to reach $\mathcal E_{\eta,\delta}$.
Sec.~\ref{sec: linear convergence rate} addresses this remaining step.

\subsection{Global Linear Convergence Rates}
\label{sec: linear convergence rate}

Now, we focus on entering the neighborhood $\mathcal E_{\eta,\delta}$.
The almost sure convergence rate established in Thm.~\ref{thm: sublinear} implies that the number of iterations needed  
\[
    T_{\eta,\delta} \coloneqq \inf\{k\geq0:\varepsilon_k\leq\eta\delta\},
\]
is almost surely finite. 
We first study the tails of its distribution.
\begin{proposition}
    \label{eq: global T}
    Let $d\geq3$, $\eta\in(0,1)$, 
    and let $q$ be as in Thm.~\ref{thm: local}.
    There exists 
    \begin{equation}
        \label{eq: p}
        p = p(\eta,\delta,\Delta)
        \geq \tfrac{1}{2\sqrt{\pi d}} \left(\tfrac{\eta\delta}{\Delta}\right)^{\frac{d-2}{2}}
        > 0
    \end{equation}
    such that
    \begin{equation}
        \label{eq: tail}
        \PP(T_{\eta,\delta}>k) \leq (1-p)^k 
        \qquad k\geq0.
    \end{equation}
    In particular, 
    $T_{\eta,\delta}<\infty$ almost surely. 
\end{proposition}
\begin{proof}
    By Bayes rule, we have
    \begin{align}
    \PP(T_{\eta,\delta}>k+1) 
        & = \int_{v^k \in \sphere^{d-1}} \PP(T_{\eta,\delta}>k+1 \mid V^k = v^k) \, \dd \PP_{V^k}( v^k) \nonumber \\
        & = \int_{v^k \in \mathcal E_{\eta, \delta}} \PP(T_{\eta,\delta}>k+1 \mid V^k = v^k) \, \dd \PP_{V^k}( v^k)  \label{eq: Bayes} \\
        & \quad + \int_{v^k \notin \mathcal E_{\eta, \delta}} \PP(T_{\eta,\delta}>k+1 \mid V^k = v^k) \, \dd \PP_{V^k}( v^k). \nonumber 
    \end{align}
    Since $V^k \in \mathcal E_{\eta,\delta}$ is equivalent to $T_{\eta,\delta} \le k$, the first integral is zero. Next, we focus on the complement  
    \begin{equation}
        \label{eq: counter}
        \PP(T_{\eta,\delta} = k+1 \mid V^k = v^k)
        = 1 - \PP(T_{\eta,\delta}>k+1 \mid V^k = v^k) 
    \end{equation}
    for $v^k \notin \mathcal E_{\eta, \delta}$ characterized by the random variable $X^k$.

    Let $u_1 \in \sphere^{d-1}$ be the eigenvector of $A^*A$ corresponding to $\lambda_1$.
    We write $\alpha_k \coloneqq \langle v^k,u_1\rangle$, 
    and define
    \[
        w^k \coloneqq \frac{u_1 - \alpha_k v^k}{\sqrt{1-\alpha_k^2}}
        \in \sphere^{d-1}.
    \]
    Notice that the denominator is nonzero, 
    since $|\alpha_k| = 1$ would imply $\varepsilon_k=0$ and contradict $v^k \notin \mathcal E_{\eta, \delta}$.
    Furthermore,
    let $P_{\operatorname{span}\{v^k, x^k\}}$ be the orthogonal projection onto $\operatorname{span}\{v^k, x^k\}$
    such that 
    \begin{equation*}
        \label{eq: proj}
        P_{\operatorname{span}\{v^k, x^k\}} u_1
        = \scp{u_1}{v^k} v^k + \scp{u_1}{x^k} x^k.
    \end{equation*}
    Orthogonality gives
    \[
        \norm{P_{\operatorname{span}\{v^k, x^k\}} u_1}^2
        = \abs{\scp{u_1}{v^k}}^2 + \abs{\scp{u_1}{x^k}}^2,
    \]
    and by $u_1 = \alpha_k v^k + \sqrt{1 - \alpha_k^2} w^k$
    we obtain 
    $\scp{u_1}{x^k} = \sqrt{1 - \alpha_k^2} \scp{w^k}{x^k}$
    such that
    \begin{equation}
        \label{eq: x w}
        \|P_{\operatorname{span}\{v^k, x^k\}} u_1\|^2
        =\alpha_k^2+(1-\alpha_k^2)|\langle x^k,w^k\rangle|^2
        \geq|\langle x^k,w^k\rangle|^2,
    \end{equation}
    since $\abs{\scp{x^k}{w^k}} \leq 1$ by the Cauchy-Schwarz inequality.
    On the event
    \[
        \mathcal G_k \coloneqq \left\{x \in \sphere^{d-1} \cap T_{v^k} 
        \;\middle|\; 
        |\langle x,w^k\rangle|^2
                      \geq 1-\frac{\eta\delta}{\Delta}\right\},
    \]
    the vector $z \coloneqq P_{\operatorname{span}\{v^k, x^k\}} u_1/\|P_{\operatorname{span}\{v^k, x^k\}} u_1\|$ is well defined and satisfies 
    \[
        |\langle z,u_1\rangle|^2 
        = |\langle z, P_{\operatorname{span}\{v^k, x^k\}}u_1\rangle|^2 
        =  \|P_{\operatorname{span}\{v^k, x^k\}}u_1\|^2
        \overset{\eqref{eq: x w}}{\geq} \abs{\scp{x^k}{w^k}}^2
        \geq 1 - \frac{\eta \delta}{\Delta}.
    \]
    Since $\norm{z} = 1$, expanding $z$ in an orthonormal eigenbasis $\{u_i\}_{i = 1}^d$ of $A^*A$ gives
    \begin{equation}
        \label{eq: err}
        \varepsilon(z)
        = \lambda_1 - \|Az\|^2
        = \sum_{i=2}^d(\lambda_1-\lambda_i)
        \abs{\scp{z}{u_i}}^2
        \leq \Delta \sum_{i=2}^d \abs{\scp{z}{u_i}}^2
        = \Delta (1-\abs{\scp{z}{u_1}}^2)
        \leq\eta\delta.
    \end{equation}
    Hence, by optimality from \eqref{eq: plane}, 
    on $\mathcal G_k$ it holds 
    \begin{align*}
        \varepsilon_{k+1}
        = \lambda_1 - \norm{A v^{k+1}}^2
        = \min_{v \in \operatorname{span}\{v^k, x^k\} \cap \sphere^{d-1}} \lambda_1 - \norm{A v}^2
        \leq \lambda_1 - \norm{A z}^2
        \leq \eta\delta.
    \end{align*}
    Therefore,
    \begin{equation}\label{eq: tail lower bound}
    \PP(T_{\eta,\delta} = k+1 \mid V^k = v^k) \ge \PP(X^k \in \mathcal G_k \mid V^k = v^k) = \PP_{X^k \mid V^k=v^k}(\mathcal G_k)
    \end{equation}
    and it remains to bound the conditional probability of $X^k \in \mathcal G_k$. Recall that given $V^k=v^k$ the random vector $X^k$ is uniformly distributed on the sphere \(T_{v^k} \cap \sphere^{d-1} \cong \sphere^{d-2}\). 
    Since \(w^k \in T_{v^k} \cap \sphere^{d-1}\), 
    the event \smash{\(\mathcal G_k = \bigl\{|\langle X^k,w^k\rangle|^2\ge 1-\frac{\eta\delta}{\Delta} \bigr\}\)} 
    consists of two antipodal spherical caps with polar angle \smash{\(\theta = \arccos\bigl((1-\frac{\eta\delta}{\Delta})^{1/2}\bigr)\)}.  
    According to \cite{frankl1990some,li2011concise}, their normalized surface area is given by the regularized incomplete Beta function $\mathrm{I}_x(a,b)$ evaluated at 
    $x = \sin^2(\theta) = 1 - (1-\frac{\eta\delta}{\Delta}) = \tfrac{\eta\delta}{\Delta}$, $a = \tfrac{d-2}{2}$ and $b = \tfrac{1}{2}$, yielding
    \[
        \PP_{X^k \mid V^k=v^k}(\mathcal G_k)
        = \mathrm{I}_{\tfrac{\eta\delta}{\Delta}}\!\left(\frac{d-2}{2},\frac12\right)  =\frac{2\int_0^\theta\sin^{d-3}\varphi\,d\varphi}{\int_0^\pi\sin^{d-3}\varphi\,d\varphi}.
    \]
    By \cite[Lem.~A.1]{bresch2024matrixfreestochasticcalculationoperator},
    we obtain the more restrictive bound 
    \[
        \PP_{X^k \mid V^k=v^k}(\mathcal G_k)
        =: p(\eta,\delta,\Delta)
        \geq \tfrac{1}{2\sqrt{\pi d}} \left(\tfrac{\eta\delta}{\Delta} \right)^{\frac{d-2}{2}}
        > 0,
    \]
    which proves \eqref{eq: p}. 
    Now, we return to \eqref{eq: Bayes} and substitute \eqref{eq: counter} and \eqref{eq: tail lower bound} to get
    \begin{align*}
        \PP(T_{\eta,\delta}>k+1) 
        & = \int_{v^k \not\in \mathcal E_{\eta, \delta}} \left(1 - \PP(T_{\eta,\delta}= k+1 \mid V^k = v^k) \right) \, \dd \PP_{V^k}( v^k)  \\
        &\le \int_{v^k \not\in \mathcal E_{\eta, \delta}} \left(1 - p \right) \, \dd \PP_{V^k}( v^k) \\
        &= (1 - p) \PP_{V^k}(\mathcal E^c_{\eta, \delta})
        = (1 - p) \PP(T_{\eta,\delta}>k). 
    \end{align*}
    Repeating the bound $k$ times gives
    \[
    \PP(T_{\eta,\delta}>k+1) \le (1 - p)^{k+1} \PP(T_{\eta,\delta}>0) \le (1-p)^{k +1},
    \]
    where we used the fact that the probability is at most one. 
\end{proof}

\begin{remark}
    The initial survival probability can be bounded more sharply.
    In the last step of Prop.~\ref{eq: global T}, 
    we derive the stronger upper bound on $\PP(T_{\eta,\delta}>0)$.
    For $V^0 \sim \mathcal U(\sphere^{d-1})$, 
    we have by \eqref{eq: err} that
    \[
        \bigl\{|\langle V^0,u_1\rangle|^2\ge1-\tfrac{\eta\delta}{\Delta}\bigr\}
        \subseteq\{T_{\eta,\delta}=0\}.
    \]
    The event on the left consists of two antipodal spherical caps, now, on the $(d-1)$-dimensional sphere.
    Hence, by \eqref{eq: counter}, we obtain 
    \begin{equation*}
        \PP(T_{\eta,\delta} > 0)
        = 1 - \PP(T_{\eta,\delta} = 0)
        \leq 1 - \mathrm I_{\frac{\eta\delta}{\Delta}}\!\left(\frac{d-1}{2}, \frac{1}{2}\right)
        \leq 1 - \tfrac{1}{2 \sqrt{\pi(d+1)}}\!\left(\tfrac{\eta\delta}{\Delta}\right)^{\frac{d-1}{2}}.
    \end{equation*}
\end{remark}

Now, we are ready to establish the linear convergence rates from the main theorem, cf.~Thm.~\ref{thm: contribution}.

\begin{theorem}
    \label{thm: global}
    Consider the setting as in Prop.~\ref{eq: global T}.
    For all integers $0\leq j\leq k$, we have
    \begin{equation}
        \label{eq: global-bound}
        \EE[\varepsilon_k]
        \leq \Delta(1-p)^j+\eta\delta\,q^{k-j}.
    \end{equation}
    Consequently, 
    with 
    \begin{equation}
        \varrho_1 \coloneqq \varrho_1(\eta, \delta,\Delta) \coloneqq 
        \exp\left( \frac{\log(1-p)\log(q)}{\log(1-p) + \log(q)} \right) < 1    
    \end{equation}
    it holds
    \begin{equation}
        \label{eq: global-geometric}
        \EE[\varepsilon_k]
        \leq C \varrho_1^k,
        \quad \text{where} \quad 
        C \coloneqq \Delta + \eta\delta q^{-1}.
    \end{equation}
\end{theorem}
    
\begin{proof}
    For $0\leq j\leq k$, 
    we split
    \[
        \EE[\varepsilon_k]
        =\EE[\varepsilon_k\mathbf1_{\{T_{\eta,\delta}>j\}}]
        + \EE[\varepsilon_k\mathbf1_{\{T_{\eta,\delta} \leq j\}}].
    \]
    By $\varepsilon_k \le \Delta$ and Prop.~\ref{eq: global T}, the first term is at most 
    \begin{equation}\label{eq: exp not in neighborhood}
        \EE[\varepsilon_k\mathbf1_{\{T_{\eta,\delta}>j\}}] 
        \le \EE[\Delta \mathbf1_{\{T_{\eta,\delta}>j\}}]
        = \Delta \PP(T_{\eta,\delta}>j)
        \le \Delta (1-p)^j.   
    \end{equation}
    The event $\{T_{\eta,\delta}\leq j\}$ implies $\varepsilon_j\leq\eta\delta$. Therefore, taking the expectation over $X_{j}, \ldots, X_{k}$ first by Thm.~\ref{thm: local} and \eqref{eq: local-rate} in particular gives
    \begin{align}
        \EE[\varepsilon_k\mathbf1_{\{T_{\eta,\delta}\leq j\}}]
        & = \EE[ \EE_{X_j, \ldots, X_k} [\varepsilon_k \mathbf1_{\{T_{\eta,\delta}\leq j\}}]]
        = \EE[ \EE_{X_j, \ldots, X_k} [\varepsilon_k ] \mathbf1_{\{T_{\eta,\delta}\leq j\}}] \nonumber \\
        & \le q^{k-j} \EE[\varepsilon_j\mathbf1_{\{T_{\eta,\delta}\leq j\}}]
        \le \eta \delta q^{k-j} \EE[\mathbf1_{\{T_{\eta,\delta}\leq j\}}]
        \le \eta \delta q^{k-j}. \label{eq: exp in neighborhood}
    \end{align}
    This proves~\eqref{eq: global-bound}. 
    Now, we would ideally choose $j$ such that $(1-p)^{j} = q^{k-j}$. 
    This corresponds to $(q(1-p))^{j} = q^{k}$ and taking the logarithm gives optimal $j_*$
    \[
    j_* = k \log(q) / (\log(1-p) + \log(q)).
    \]
    Since this value is not necessarily an integer, we pick $j = \lceil j_* \rceil$, the rounded-up value. As a result, we have
    \[
    \EE[\varepsilon_k] \le \Delta (1-p)^{\lceil j_* \rceil} + \eta \delta q^{k - \lceil j_* \rceil}
    \le \Delta (1-p)^{j_*} + \eta \delta q^{-1} q^{k - j_*},    
    \]
    where we used that $\lceil j_* \rceil -1 \le j_* \le \lceil j_* \rceil$. Finally, we observe that
    \[
    q^{k - j_*} = (1-p)^{j_*} = (1-p)^{k \log(q) / (\log(1-p) + \log(q))} = \varrho_1^k,
    \]
    which proves~\eqref{eq: global-geometric}.
\end{proof}

We also establish an alternative linear rate in the case when $q \neq 1 - p$.

\begin{theorem}
    \label{thm: global-unequal}
    Consider the setting as in Thm.~\ref{thm: global}.
    In the case $q \neq 1 - p$ we have   
    \begin{equation}
        \label{eq: global-geometric-strong neq}
        \EE[\varepsilon_k]
        \leq C \varrho_2^k,
        \quad \text{where} \quad 
        C \coloneqq \Delta + \eta\delta \bigl(1 + \tfrac{1}{|1 - q - p|}\bigr),
        \; 
        \varrho_2 \coloneqq \varrho_2(\eta, \delta,\Delta) \coloneqq \max\{q, 1 - p\},
    \end{equation}
    and otherwise, if $q = 1- p$, we get a near-linear rate
    \begin{align}
        \label{eq: global-geometric-strong eq}
        \EE[\varepsilon_k]
        \leq 
        (\Delta  + \eta \delta (1 + k \varrho_2^{-1} )) \varrho_2^{k},
        \quad k \geq 1,
    \end{align}
    that can be further bounded by
    \[
    \EE[\varepsilon_k]
        \le\left(\Delta+\eta\delta+
        \frac{\eta\delta}{e\alpha\varrho_2\log(1/\varrho_2)}\right)
        \varrho_2^{(1-\alpha)k},
        \quad k \geq 1, 
        \quad \text{for any} \quad \alpha \in (0,1).
    \]
\end{theorem}
\begin{proof}
    In contrast to Thm.~\ref{thm: global}, we perform a more detailed split on $T_{\eta,\delta}$,
    more precisely
    \[
        \EE[\varepsilon_k]
        =\EE[\varepsilon_k\mathbf1_{\{T_{\eta,\delta}>k\}}]
        + \sum_{t = 0}^k \EE[\varepsilon_k\mathbf1_{\{T_{\eta,\delta} = t\}}].
    \]
    The first term is bounded as in \eqref{eq: exp not in neighborhood}. The second term can be bounded analogously to \eqref{eq: exp in neighborhood} by $\EE[\varepsilon_k\mathbf1_{\{T_{\eta,\delta} = t\}}] \le \eta \delta q^{k-t} \PP(T_{\eta,\delta} = t)$. Using
    \[
        \PP(T_{\eta,\delta} = t)
        \leq \PP (T_{\eta,\delta} > t-1)
        \leq (1 - p)^{t-1}.
    \]
    we obtain 
    \begin{align}
        \EE[\varepsilon_k]
        &\leq 
        \Delta (1 - p)^k 
        + \eta \delta \sum_{t = 0}^k q^{k-t} \PP(T_{\eta,\delta} = t) \nonumber \\[-8pt]
        &\leq 
        \Delta (1 - p)^k 
        + \eta \delta q^{k}
        + \eta \delta q^{k-1}\sum_{t = 1}^k q^{-(t-1)} (1 - p)^{t-1}.
        \label{eq: sum exp}
     \end{align}
    If $q \neq (1-p)$, we get a geometric sum \eqref{eq: sum exp}, which gives   
    \begin{align*}
        \EE[\varepsilon_k]
        & \le \Delta (1 - p)^k 
        + \eta \delta q^{k}
        + \eta \delta q^{k-1} \frac{q^{-k}(1 - p)^{k} - 1}{q^{-1}(1 - p)-1}.
    \end{align*}
    Note that
    \begin{align*}
    q^{k-1} \frac{q^{-k}(1 - p)^{k} - 1}{q^{-1}(1 - p)-1} 
    = \frac{q^{k} - (1 - p)^{k}}{q - (1 - p)} 
    = \frac{\max^{k}\{q, 1 - p\} - \min^{k}\{q, 1-p\}} {|q - (1 - p)|} \le \frac{\varrho_2^{k}}{|1 - q - p|},
    \end{align*}
    which gives 
    \begin{equation*}
        \EE[\varepsilon_k]
        \leq 
        \Bigl(\Delta + \eta\delta \bigl(1 + \tfrac{1}{|1 - q - p|}\bigr)\Bigr)
        \varrho_2^{k}.
    \end{equation*}
    Now, if $q = (1-p) = \varrho_2$, the sum is exactly $k$, yielding  
    \begin{align*}
        \EE[\varepsilon_k]
        & \le
        \Delta (1 - p)^k 
        + \eta \delta q^{k}
        + \eta \delta k q^{k-1}
        = (\Delta  + \eta \delta (1 + k \varrho_2^{-1} )) \varrho_2^{k}.
    \end{align*}
    Splitting $\varrho_2^{k} = \varrho_2^{k(1-\alpha)} \varrho_2^{k\alpha}$ for any $\alpha \in (0,1)$ gives
    \begin{equation}\label{eq: alpha split}
    \EE[\varepsilon_k]
        \le (\Delta  + \eta \delta (1 + k \varrho_2^{k\alpha}  \varrho_2^{-1} )) \varrho_2^{k(1-\alpha)}.
    \end{equation}
    Consider the function $g(t) = t \varrho_2^{\alpha t}$.
    Since $g(t) \to 0$, as $t \to \infty$, $g(0) =0$ and $g(t) \ge 0$ on $[0, +\infty)$, the function is bounded on $[0, +\infty)$ and its maximum value is a critical point
    \begin{align*}
        0 &= \frac{\dd}{\dd t} \bigl(t \varrho_2^{\alpha t}\bigr)
        = \varrho_2^{\alpha t} + \alpha t \log(\varrho_2) \varrho_2^{\alpha t} 
        \; \Leftrightarrow \; 
        0 = \varrho_2^{\alpha t} (1 - \alpha t \log(1/\varrho_2))
        \; \Leftrightarrow \; 
        t = \frac{1}{\alpha\log(1/\varrho_2)}.
    \end{align*}
    Consequently, we bound 
    \[
    k \varrho_2^{k\alpha} = g(k) \le \max_{t \ge 0} g(t) = \frac{1}{\alpha\log(1/\varrho_2)} \varrho_2^{-\frac{1}{\log(\varrho_2)}}.
    \]
    Substituting this bound into \eqref{eq: alpha split} concludes the proof.    
\end{proof}

\begin{figure}
    \begin{minipage}{0.7\linewidth}
        \includegraphics[width=\linewidth,clip=true, trim=0pt 55pt 0pt 0pt]{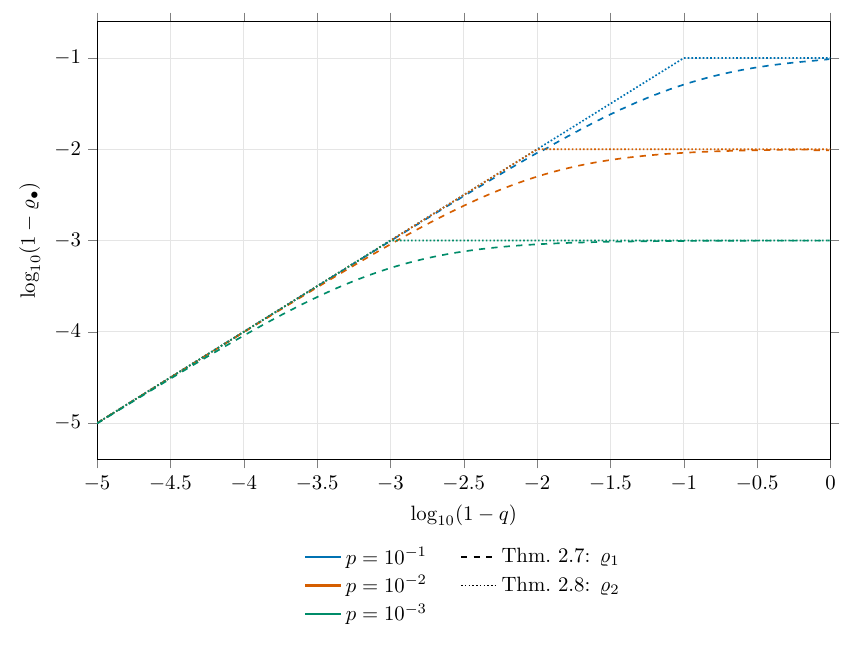}%
    \end{minipage}%
    \begin{minipage}{0.3\linewidth}
        \includegraphics[width=0.7\linewidth,clip=true, trim=135pt 0pt 190pt 260pt]{convergence_rates_log.pdf}
        \includegraphics[width=\linewidth,clip=true, trim=210pt 0pt 80pt 260pt]{convergence_rates_log.pdf}
    \end{minipage}
    \caption{%
    Comparison of global geometric convergence factors $\varrho_1$ and $\varrho_2$
    on the local contraction factor $q$, for $p \in \{0.1,0.01,0.001\}$. The larger $\log_{10}(1-\varrho_\bullet)$, the faster is convergence.
    }
    \label{fig: convergence rates}
\end{figure}

The global bound accounts for the time needed to enter the superlevel set. Its cap probability $p(\eta,\delta,\Delta)$ can be small for large $d$, as we see an explicit dependence on the dimension in \eqref{eq: p}. Our numerical trials indicate that the lower bound on $p(\eta,\delta,\Delta)$ in \eqref{eq: p} is significantly suboptimal and the rate is much faster. 

Fig.~\ref{fig: convergence rates} shows the global rates $\varrho_1$ from Thm.~\ref{thm: global} and $\varrho_2$ from Thm~\ref{thm: global-unequal} over the local convergence rate $q$ for various $p$. We observe that rates $\varrho_1$ and $\varrho_2$ are close if $|1-p - q|$ is large and $\varrho_2$ is better than $\varrho_1$ for small $|1-p - q|$.

Next, we establish bounds for the expected error and the distance to the leading eigenvector in terms of errors, which allows us to extend linear convergence rates to these quantities.  

\begin{corollary}\label{eq: other-errors}
    Consider the setup of Thm.~\ref{thm: contribution}. 
    For every $k\geq0$, we almost surely have
    \begin{align}
        0\leq \sqrt\lambda_1-\|Av^k\|&\leq\frac{\varepsilon_k}{\sqrt\lambda_1},
        \label{eq: norm-error}\\
        1-|\langle v^k,u_1\rangle|^2&\leq\frac{\varepsilon_k}{\delta},
        \label{eq: angle-error}\\
        \min_{s\in\{-1,1\}}\|v^k-su_1\|^2&\leq\frac{2\varepsilon_k}{\delta},
        \label{eq: vector-error}\\
        \varepsilon_k(\delta - \varepsilon_k)
        \leq 
        \|r^k\|^2&\leq \varepsilon_k(\Delta - \varepsilon_k)
        \leq \Delta \varepsilon_k.
        \label{eq: residual-upper}
    \end{align}
\end{corollary}
\begin{proof}
    The first assertion follows from 
    \[
    \varepsilon_k= \lambda_1 - \norm{Av^k}^2= \bigl(\sqrt{\lambda_1}+\|Av^k\|\bigr) \bigl(\sqrt{\lambda_1}-\|Av^k\|\bigr)
    \ge \sqrt{\lambda_1} (\sqrt{\lambda_1}-\|Av^k\|).
    \]
    Expansion of $v^k$ in an eigenbasis $\{u_i\}_{i = 1}^d$ of $A^*A$ yields
    \[
        \varepsilon_k
        = \lambda_1 - \norm{A v^k}^2
        = \sum_{i = 1}^d (\lambda_1 - \lambda_i) \abs{\scp{u_i}{v^k}}^2
        \geq \delta \sum_{i = 2}^d \abs{\scp{u_i}{v^k}}^2
        = \delta(1-|\langle v^k,u_1\rangle|^2),
    \]
    and hence \eqref{eq: angle-error}.
    Furthermore,
    \[
        \min_{s\in\{-1,1\}}\|v^k-su_1\|^2
        =2(1-|\langle v^k,u_1\rangle|) 
        \leq2(1-|\langle v^k,u_1\rangle|^2),
    \]
    yielding \eqref{eq: vector-error}.
    The last inequality is exactly the statement of Lem.~\ref{lem: residual}.
\end{proof}

Therefore, combining Cor.~\ref{eq: other-errors} with Thm.~\ref{thm: global} or Thm.~\ref{thm: global-unequal} gives a linear convergence rate for the average norm of the residual and the average distance to the leading eigenvector.

\subsection{Linear Convergence for Estimating the Generalized Rayleigh Quotient}\label{sec: generalized quotient}

In this section, we briefly discuss the extension of the derived linear convergence rates to the generalized-Rayleigh-quotient method proposed in \cite{bresch2025stochastic}.
We consider for the computation of the extremal generalized Rayleigh quotient 
\begin{equation}
    \label{eq: gen rayleigh}
    \mathcal R(A,B) 
    = \max_{v \in \RR^{d}, v \neq 0} \frac{\scp{v}{Av}}{\scp{v}{Bv}}
    = \max_{v \in \RR^{d}, v \neq 0} \frac{\scp{v}{\herA v}}{\scp{v}{Bv}},
    \qquad \herA \coloneqq \frac{1}{2}(A + A^{*}),
\end{equation}
of $A \in \RR^{d \times d}$ and positive definite $B = B^{*} \in \RR^{d \times d}$.
To solve \eqref{eq: gen rayleigh}, the authors of \cite{bresch2025stochastic} consider the exact random-plane iteration
\[
    v_{k+1} \in \!\!\!\!\argmax_{\substack{z\in\operatorname{span}\{v^k,x^k\}\\ \langle z, Bz\rangle=1}}
    \!\!\!\scp{z}{\herA z},
\]
with
\[
    \PP_{X^k \mid V^k} = \mathcal U\!\left(\sphere^{d-1}\cap T_{B v^k}\right),
    \qquad
    T_{Bv^k} \coloneqq \{w \in \RR^d \mid \scp{w}{B v^k} = 0\},
\]
similar to \eqref{eq: plane}.
Hence, we obtain a random sequence $(v^k)_{k \in \NN}$ with $\scp{v^k}{B v^k} = 1$,
where the latter maximization problem can be solved uniquely \cite[Thm.~3.1]{bresch2025stochastic}, almost surely.
Transforming $\tilde v^k \coloneqq B^{1/2}v^k$, and constructing 
\[
    \tilde x^k 
    \coloneqq \frac{B^{1/2} x^k}{\norm{B^{1/2} x^k}}
    = \frac{B^{1/2} x^k}{\sqrt{\scp{x^k}{B x^k}}},
\]
the transformed search direction
fulfills
\[
    \norm{\tilde v^k}
    = \norm{\tilde x^k}
    = 1,
    \qquad 
    \scp{\tilde x^k}{\tilde v^k}
    = \frac{\scp{v^k}{B x^k}}{\sqrt{\scp{x^k}{B x^k}}}
    = 0
\]
and is generally not uniformly distributed.
Moreover, it holds
\[
    B^{1/2}\operatorname{span}\{v^k,x^k\}
    =\operatorname{span}\{\tilde v^k,\tilde x^k\},
    \qquad
    z^*\herA z=(B^{1/2}z)^*\underbracket{B^{-1/2}A^HB^{-1/2}}_{=: C}(B^{1/2}z).
\]
Consequently, 
the generalized Rayleigh quotient \eqref{eq: gen rayleigh} becomes the ordinary Rayleigh quotient of $C \coloneqq B^{-1/2}A^HB^{-1/2}$,
and the transformed iteration satisfies
\[
    \tilde v^{k+1} \in \argmax_{\substack{z\in\operatorname{span}\{\tilde v^k, \tilde x^k\}\\\|z\|=1}}
    \scp{z}{Cz}.
\]
Moreover, the eigenvalues of $C$ and the generalized eigenvalues of $(\herA, B)$ given by
\begin{equation*}
    \{\lambda \in \RR : \exists v \in \RR^d \setminus\{0\}, \herA v = \lambda B v\}
    = \spec\{C\}  
    =: \{\tilde \lambda_1, ..., \tilde \lambda_d\}
\end{equation*}
coincide.
For that reason, we redefine in the manner of \eqref{eq: coefs}, \eqref{eq: eps v}, and \eqref{eq: res}
\begin{equation*}
    \rho_k\coloneqq \scp{\tilde v^k}{C\tilde v^k}
    =\frac{\scp{v^k}{Av^k}}{\scp{v^k}{Bv^k}},
    \qquad
    \varepsilon_k\coloneqq\tilde\lambda_1-\rho_k,
\end{equation*}
and
\begin{equation*}
    r^k\coloneqq C\tilde v^k-\rho_k\tilde v^k
    =B^{-1/2}(\herA v^k-\rho_kBv^k).
\end{equation*}
Thereby, the system matrix $H_k$ from \eqref{eq: coefs} is defined by 
$a_k\coloneqq \langle C \tilde v^k,\tilde x^k\rangle$,
and $s_k\coloneqq \langle \tilde x^k, C \tilde x^k\rangle.$
The proof of Lem.~\ref{lem: ascent} carries over through \eqref{eq: ak}. 
Using the orthogonal projection with respect to the changed Riemannian metric
yields 
\[
    P_{Bv} \coloneqq I-\frac{(Bv)(Bv)^*}{\|Bv\|^2}
    \qquad \text{with} \qquad 
    \EE_{x^k \sim X^k \mid V^k = v^k}[x^k(x^k)^*] = \frac{P_{Bv^k}}{d-1}.    
\]
Then, the conditional second moment can be, instead of \eqref{eq: exp ak}, bounded by 
\[
    \EE_{X^k \mid V^k = v^k}\left[|\langle \tilde x^k,r^k\rangle|^2\right]
    =
    \EE_{X^k \mid V^k = v^k}\left[\frac{|\langle x^k,B^{1/2}r^k\rangle|^2}{\scp{x^k}{Bx^k}}\right]
    \ge
    \frac{\|P_{Bv^k}B^{1/2}r^k\|^2}
    {(d-1)\lambda_{\max}(B)}.
\]
To bound the numerator, we have
\begin{align*}
    \|P_{Bv^k}B^{1/2}r^k\|^2
    &=\min_{\gamma\in\mathbb R}
    \|B^{1/2}r^k-\gamma Bv^k\|^2
    =\min_{\gamma\in\mathbb R}
    \|B^{1/2}(r^k-\gamma \tilde v^k)\|^2 \\
    &\ge
    \lambda_{\min}(B)
    \min_{\gamma\in\mathbb R}\|r^k-\gamma \tilde v^k\|^2
    =\lambda_{\min}(B)\|r^k\|^2.
\end{align*}
Together, we obtain
\[
    \EE_{X^k \mid V^k = v^k}\left[
    |\langle \tilde x^k,r^k\rangle|^2\right]
    \ge\frac{\|r^k\|^2}{(d-1)\kappa(B)},
    \qquad 
    \kappa(B) \coloneqq \frac{\lambda_{\max}(B)}{\lambda_{\min}(B)}.
\]
Combining this with \eqref{eq: pathwise-ascent} proves the generalized version of Lem.~\ref{lem: ascent}
\begin{equation}
    \label{eq: ascent-gen}
    \EE_{X^k \mid V^k = v^k}[\varepsilon_{k+1}]
    \le
    \varepsilon_k-
    \frac{\|r^k\|^2}{(d-1)\kappa(B)\Delta}.
\end{equation}
The change of variables applies directly to the proof technique of Lem.~\ref{lem: residual}, hence we obtain 
\[
    \varepsilon_k(\delta-\varepsilon_k)
    \le\|r^k\|^2
    \le\varepsilon_k(\Delta-\varepsilon_k)
    \le\Delta\varepsilon_k.
\]
Together with \eqref{eq: ascent-gen} we obtain  for $\varepsilon_k \leq \eta\delta$ that
\[
    \EE_{X^k \mid V^k = v^k}[\varepsilon_{k+1}]
    \le
    q_B \varepsilon_k,
    \qquad \text{where} \qquad
    q_B \coloneqq 1-\frac{(1-\eta)\delta}{(d-1)\kappa(B)\Delta},
\]
which provides the generalized local linear rate constant \eqref{eq: q}.

Similar to Prop.~\ref{eq: global T}, 
let $\tilde u_1$ be a leading unit eigenvector of $C$.
The eigenbasis argument in \eqref{eq: err} gives
\[
    \tilde \lambda_1 - \scp{z}{\herA z}
    \le\Delta\bigl(1-|\langle B^{1/2} z , \tilde u_1\rangle|^2\bigr),
    \qquad \|z\|_B \coloneqq \sqrt{\langle z, B z \rangle} = 1,
\]
and
\[
    1-|\langle B^{1/2} z,\tilde u_1\rangle|^2
    \le2\bigl(1-|\langle B^{1/2} z,\tilde u_1\rangle|\bigr)
    =\!\!\!\min_{\sigma\in\{-1,1\}}\!\!\|B^{1/2} z-\sigma \tilde u_1\|^2
    =\!\!\!\min_{\sigma\in\{-1,1\}}\!\!\|z-\sigma B^{-1/2} \tilde u_1\|_B^2.
\]
Consequently, with $\varepsilon \coloneqq \sqrt{\eta\delta/\Delta}$ 
it holds 
\[
    \left\{z \in \RR^d : \|z\|_B=1,\ 
    \min_{\sigma\in\{-1,1\}}\|z-\sigma B^{-1/2} \tilde u_1\|_B < \varepsilon
    \right\}
    \subseteq\left\{\varepsilon(z) \coloneqq \tilde \lambda_1 - \scp{z}{\herA z}\le\eta\delta\right\}.
\]
By \cite[Lem.~5.5]{bresch2025stochastic}
the entrance probability becomes
\[
    p_{B}
    \coloneqq
    \mathrm I_{r_{B}^2}\Bigl(\frac{d-2}{2}, \frac{1}{2}\Bigr),
    \quad
    \text{where}
    \quad
    r_{B}
    \coloneqq 
    \frac{\frac{\eta\delta}{\Delta}}{16 \kappa(B)^{1/2}\sqrt{1 - \tfrac{\eta\delta}{32\Delta}}}.
\]
Then, Theorems~\ref{thm: global} and~\ref{thm: global-unequal} 
apply unchanged with $q$ and $p$ replaced by $q_B$ and $p_B$, respectively.

\section{Numerical Results}\label{sec:num}

In this section, 
we substantiate the theoretical results
by exploring the actual speed of the random plane iteration (\ref{eq: plane}--\ref{eq: construction}) from \cite{bresch2024matrixfreestochasticcalculationoperator},
compared with the derived convergence rates in Thm.~\ref{thm: global} and~\ref{thm: global-unequal}.

Firstly, we study the local convergence rate in the neighborhood of the solution.
Secondly, we analyze the global convergence behavior with respect to the entrance time and subsequent contraction.
Finally, we study the influence of the spectrum $\{\lambda_i\}_{i=1}^d$ of $A^*A$ on the convergence.

All algorithms are implemented in Python and the code is publicly available%
\footnote{%
\url{https://github.com/JJEWBresch/LinearConvergenceOperatorNormSolver}.}.
The experiments are performed on an off-the-shelf MacBook Pro 2020
with Apple M1 (8‑Core GPU, 3.2~GHz) and 16~GB~RAM.

\subsection{Controlled Spectra and Initialization}\label{sec: local convergence}

In this subsection, we analyze how the spectral gap, dimension, and initialization affect the convergence behavior and compare empirical and theoretical local rates.

Given a desired spectral gap $\delta \in [0,1/2]$, we generate matrices $A \in \RR^{d \times d}$ via
$$
    A = \diag\bigl(\sqrt{\lambda_i}\bigr) Q
$$
with $Q$ being sampled uniformly on the orthogonal group $\Orth(d) = \{ V \in \RR^{d \times d} \mid V^*V = VV^* =I_d \}$, denoted by $Q \sim \mathcal U(\Orth(d))$. The spectrum is set to $\lambda_1 = 1$, $\lambda_2 = 1 - \delta = \lambda_1 -  \delta$ and the lower eigenvalues evenly spaced in
$[0,1/2]$, including $\lambda_d=0$.
Consequently, $A^* A$ has spectral gap $\delta$, spectral width $\Delta=1$ and $\|A\|=1$.

Furthermore, for given $\eta \in [0,1]$ we initialize the algorithm 
at 
\begin{equation}
    \label{eq: sec1 init}
    v^0 \coloneqq \sqrt{1-\eta/2}\,u_1+\sqrt{\eta/2}\,u_2
    \in \sphere^{d-1},
\end{equation}
where $u_1$ and $u_2$ are the eigenvectors of $A^*A$
corresponding to $\lambda_1$ and, respectively, $\lambda_2$.
Hence, $\varepsilon_0=\eta \delta/2 \leq \eta\delta$, so the local convergence rate \eqref{eq: q} applies from the first iteration,
\begin{equation*}
    \label{eq: exp error}
    \EE[\varepsilon_k]
    \leq q^k\varepsilon_0,
    \qquad
    q = q(\eta, \delta, \Delta) 
    = 1-\frac{(1-\eta)\delta}{(d-1)\Delta}
    = 1 - \frac{(1-\eta)\delta}{d-1}.
\end{equation*}
Thus, the guaranteed contraction becomes weaker as the dimension increases, the gap $\delta$ decreases, or $\eta$ increases.

\begin{figure}[b!]
    \resizebox{\linewidth}{!}{%
    \begin{tabular}{@{}r@{\quad}c@{\quad}c@{}}
    &\multicolumn{2}{c}{\legline[plotblue]{}\;\;$\delta=0.5$ \qquad
    \legline[plotorange]{}\;\;$\delta=0.1$ \qquad 
    \legline[plotgreen]{}\;\;$\delta=0.03$} \\
    &\multicolumn{2}{c}{\legline[gray]{}\;\;$\eta=0.01$ \qquad 
    \legline[gray,dashed]{}\;\;$\eta=0.1$ \qquad 
    \legline[gray,dotted]{}\;\;$\eta=0.5$ \qquad
    \legline[gray,dashdotted]\;\;Optimal theoretical local bound $q_0^k$} \\
    & $d = 20$ & $d = 60$ \\
    \rotatebox{90}{\hspace{90pt}$\EE[\varepsilon_k] / \varepsilon_0$} 
    & \includegraphics[width=0.8\linewidth, clip=true, trim=20pt 340pt 470pt 20pt]{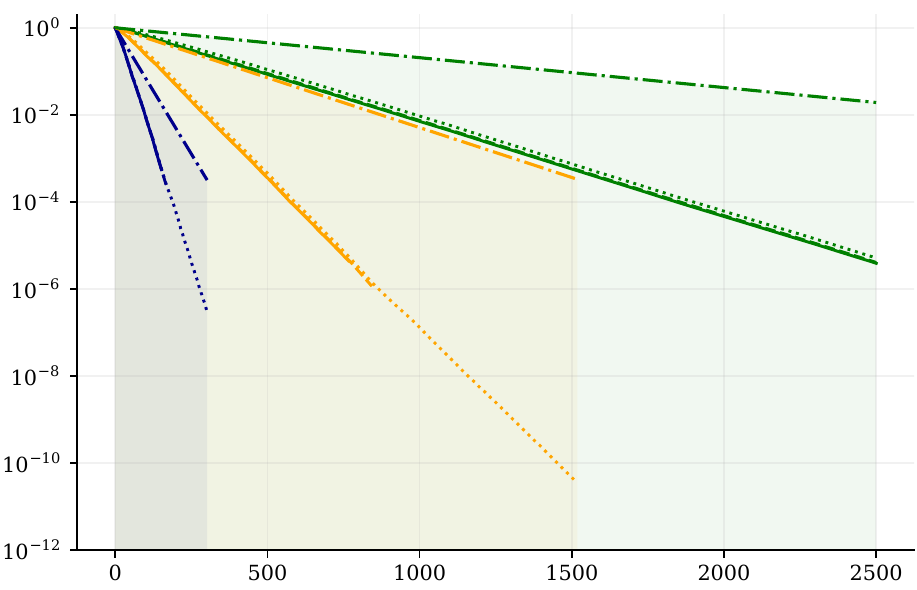} 
    & \includegraphics[width=0.8\linewidth, clip=true, trim=20pt 20pt 470pt 340pt]{local_convergence_etas_0.01_0.1_0.5.pdf} \\
    & Iteration $k$ & Iteration $k$
    \end{tabular}}
    \caption{%
    Local convergence behavior 
    for various $(d, \delta, \eta)$  
    on independent runs 
    with the initialization \eqref{eq: sec1 init}.
    The quotient of arithmetic mean error and initial error 
    over iterations is compared with the tight optimal theoretical local bound $q_0 = q(0,\delta,\Delta)$ from Ex.~\ref{example}.
    }
    \label{fig: sec1 local}
\end{figure}

In our experiments, we fix one matrix for each spectral configuration
\[
    d\in\{20,60\},\qquad
    \delta\in\{0.5,0.1,0.03\},
\]
and perform $96$ independent runs of the random-plane algorithm for each $\eta\in\{0.01,0.1,0.5\}$.
All runs start from the deterministic initial vector~\eqref{eq: sec1 init}%
---randomness arises from the sampled search directions.

For each configuration, we use at most $2\,500$ iterations and otherwise four theoretical decay times, where one decay time corresponds to $1/(-\log q)$ iterations:
\[
    K=\min\!\left\{
    \left\lceil\frac{4}{-\log q}\right\rceil,\;2500
    \right\}.
\]

Fig.~\ref{fig: sec1 local} reports the mean of the relative errors $\varepsilon_k / \varepsilon_0$ over $96$ runs as a function of iteration number and compare it with the optimal theortical local convergence rate $q_0 = q(0,\delta,\Delta)$ from Ex.~\ref{example}.

The results show approximately straight mean-error lines on logarithmic vertical axes over the observed contraction ranges, consistent with linear decay rates. Larger spectral gaps $\delta$ produce faster convergence, while larger $d$ slows it. These observations support the theorem’s qualitative prediction of local linear convergence and its dependence on the spectral gap and the dimension. The empirical rates, however, are faster than the actual theoretical bounds. Across all tested configurations, the observed rates exceed the theoretical rates by factors of approximately $1.9 - 6.3$. 
In contrast to theoretical results, we did not observe a strong empirical dependence on $\eta$.
For instance, at $d=20$ and $\delta=0.1$, the observed rate $-\log \hat q$ changes only from approximately $0.0163$ to $0.0162$ as $\eta$ increases from $0.01$ to $0.5$, while the theoretical rate $-\log q$ falls from $0.00522$ to $0.00264$.

\subsection{Global convergence}\label{sec: global convergence}

In this section, we aim to study the time and computational cost required to enter the neighborhood of local convergence.

Therefore, we use the same setting as in Sec.~\ref{sec: local convergence}, but now sample the initialization $v^0$ uniformly on the unit sphere. For each spectral configuration
\[
    (d,\delta)\in\{(20,0.5),(20,0.1),(60,0.1)\},
    \qquad
    \eta\in\{0.01,0.1,0.5\},
\]
we fix one matrix and perform $128$ independent runs with an iteration limit $K=3500$. The same initial vectors and random draws are reused across $\eta$, hence, only the threshold defining the local contraction region changes.

\begin{figure}[b!]
    \resizebox{\linewidth}{!}{%
    \begin{tabular}{@{}r@{\quad}c@{\quad}c@{\quad}c@{}}
    &\multicolumn{3}{c}{\legline[empiricalblue]{}\;\;Empirical mean \qquad
    \legband\;\;$95\%$ bootstrap interval \qquad
    \legbandred\;\;Area $95\%$ upper limit \qquad 
    \legline[black, dashed]{\times}\;\;Optimal local bound $q_0^k$
    \legline[black]{+}\;\;Global bound $\varrho_\bullet^k$}\\
    \rotatebox{90}{\hspace{60pt} $\EE[\varepsilon_k]/\varepsilon_0$} 
    & \includegraphics[width=0.55\linewidth, clip=true, trim=20pt 660pt 815pt 20pt]{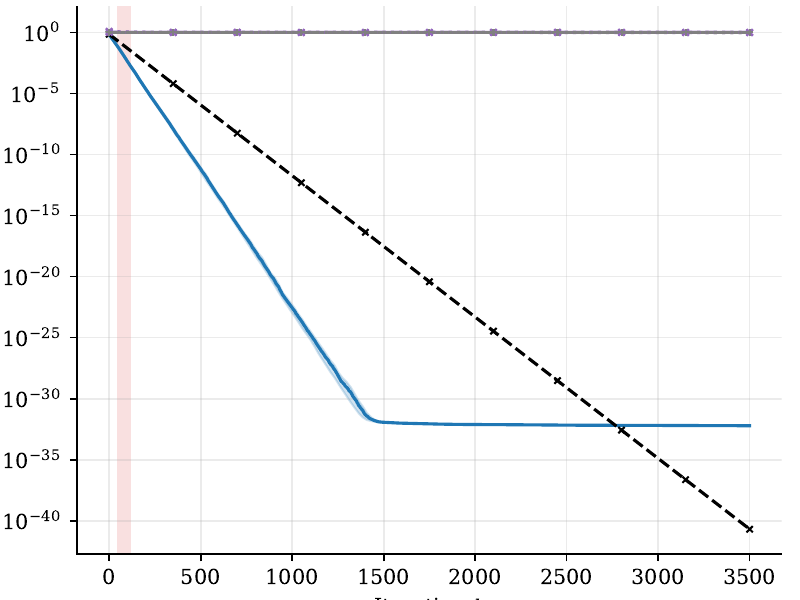} 
    & \includegraphics[width=0.55\linewidth, clip=true, trim=20pt 340pt 815pt 340pt]{global_entrance_contraction_etas_0.01_0.1_0.5.pdf}
    & \includegraphics[width=0.55\linewidth, clip=true, trim=20pt 20pt 815pt 665pt]{global_entrance_contraction_etas_0.01_0.1_0.5.pdf} \\[-8pt]
    & \tiny Iteration $k$ 
    & \tiny Iteration $k$ & \tiny Iteration $k$ \\[5pt]
    & \multicolumn{3}{c}{\legline[empiricalblue]{}\;\;Empirical survival \qquad
    \legline[gray]{}\;\;$\eta=0.01$ \qquad
    \legline[gray,dashed]{}\;\;$\eta=0.1$ \qquad
    \legline[gray,dotted]{}\;\;$\eta=0.5$ \qquad 
    \legline[gray]{+}\;\;Entrance bound $(1-p)^k$ \qquad
    \legline[red,dash dot]{}\;\;$95\%$ upper limit} \\
    \rotatebox{90}{\hspace{60pt} $\PP(T_{\eta,\delta} > k)$} 
    & \includegraphics[width=0.55\linewidth, clip=true, trim=430pt 660pt 405pt 20pt]{global_entrance_contraction_etas_0.01_0.1_0.5.pdf} 
    & \includegraphics[width=0.55\linewidth, clip=true, trim=430pt 340pt 405pt 340pt]{global_entrance_contraction_etas_0.01_0.1_0.5.pdf}
    &  \includegraphics[width=0.55\linewidth, clip=true, trim=430pt 20pt 405pt 665pt]{global_entrance_contraction_etas_0.01_0.1_0.5.pdf} \\[-8pt]
    & \tiny Iteration $k$
    & \tiny Iteration $k$
    & \tiny Iteration $k$
    \end{tabular}}
    \caption{Comparison of the spectral error,
    entrance probability over iteration,
    and the number of steps to convergence
    after entrance in the retraction region 
    for $(d,\delta)\in\{(20,0.5),(20,0.1),(60,0.1)\}$ (left to right).}
    \label{fig: sec2}
\end{figure}

The first row of Fig.~\ref{fig: sec2} shows the empirical mean error over $N=128$ independent runs, together with pointwise $95\%$ bootstrap confidence intervals (shaded), the limiting local reference $q_0$, and the theoretical global bounds. The confidence intervals are computed from $150$ bootstrap resamples, each consisting of $128$ complete trajectories sampled with replacement.

In the second row of Fig.~\ref{fig: sec2}, we compare the empirical survival probability, estimated by the fraction of runs that have not yet entered the local contraction region, with this theoretical bound. By Prop.~\ref{eq: global T}, the entrance-time survival probability is bounded by $(1-p(\eta,\delta,\Delta))^k$. The horizontal red line marks $u_N=1-0.05^{1/N}\approx0.0231$, the one-sided $95\%$ binomial upper confidence limit for the survival probability when no survivors are observed among $N$ independent runs. For each spectral configuration, the red shaded band in the first row of Fig.~\ref{fig: sec2} spans the earliest and latest iterations, across $\eta$, at which the empirical survival curves fall below this reference level. These bands summarize the observed entrance behavior across the chosen thresholds; they are neither confidence intervals for the entrance time nor ranges at which $95\%$ of runs have entered. The mean-error curves also illustrate changes in the observed decay around these entrance ranges.

\begin{wraptable}{r}{0.40\textwidth}
    \resizebox{\linewidth}{!}{%
    \begin{tabular}{c c c c c}
    \toprule
    Dimension \(d\) & Gap \(\delta\) & \(\eta=0.01\) & \(\eta=0.1\) & \(\eta=0.5\) \\
    \midrule
    20 & 0.5 & 95 & 50.5 & 22 \\
    20 & 0.1 & 255.5 & 112 & 57 \\
    60 & 0.1 & 826 & 375 & 199.5 \\
    \bottomrule
    \end{tabular}}
    \caption{Observed median entrance times for the different setting over the $N$ independent runs.}
    \label{tab: sec2}
\end{wraptable}
For the tested configurations, entrance takes longer for larger dimensions, smaller spectral gaps, and smaller values of $\eta$, which define stricter local regions. Table~\ref{tab: sec2} summarizes the sample median entrance times.
The global bounds established in Thm.~\ref{thm: global} and Thm.~\ref{thm: global-unequal} are substantially more conservative than the observed behavior. The local contraction factor $q(\eta,\delta,\Delta)$ provides a more informative description after entrance, indicating room for improvement in the global entrance-time analysis.

\subsection{Effect of the Lower spectrum}\label{sec: lower spectrum}

Lastly,
we study in this section the impact of the remaining spectrum
beyond the leading and second leading singular values,
i.e., the spectral gap,
on the achieved convergence behavior in practice.
Therefore, let us consider $\{A_{(i)}\}_{i = 1}^3$ as in Sec.~\ref{sec: local convergence},
with $\lambda_1 = 1, \delta = 0.1$, i.e. $\lambda_2 = \lambda_1 - \delta$.
For $A_1$, the remaining singular values spread over $[\sqrt\lambda_2, ..., 0]$ equidistantly, for $A_2$ they are clustered on the second-largest $\lambda_2 = ... = \lambda_{d-1} = \lambda_1 - \delta$, with $\lambda_d = 0$ and for $A_3$ we cluster them on the smallest, i.e. $\lambda_3 = ... = \lambda_d = 0$.
Therefore, all matrices have operator norm $\|A_{(i)}\|=1$, spectral gap $\delta$, and spectral width $\Delta=1$

We compare two initialization settings for each of the
three spectral configurations for $d = 100$.
In the first setting, we sample $v^0$ uniformly on the unit sphere $\sphere^{d-1}$ and perform $256$ independent runs with $20\,000$ iterations per run.
The mean spectral error relative to the initial error with the pointwise $95\%$ percentile bootstrap confidence intervals, computed using $200$ samples of the recorded trajectories, is visualized in cf.~Fig.~\ref{fig: sec4} (left).
In the second setting, for each $\eta\in\{0.01,0.1,0.5\}$, we sample random initial
vectors inside the corresponding local contraction
region $\mathcal E_{\eta,\delta}$ and perform $256$
independent runs with $5\,000$ iterations per run.
Specifically, we independently sample
\[
    w \sim \mathcal U\bigl(u_1^\perp\cap\sphere^{d-1}\bigr) = \mathcal U \bigl(\{ z \in \sphere^{d-1} \mid \scp{z}{u_1} = 0\}),
    \qquad u \sim \mathcal U(0,1),
    \qquad s \sim \mathcal U\{-1,1\},
\]
and set
\[
    L(w)\coloneqq
    \sum_{i=2}^{d}(\lambda_1-\lambda_i)
    |\langle w,u_i\rangle|^2,
    \qquad
    t\coloneqq\frac{u\eta\delta}{L(w)},
    \qquad
    v^0\coloneqq s\sqrt{1-t}\,u_1+\sqrt{t}\,w.
\]
This construction gives
$\varepsilon_0=u\eta\delta\leq\eta\delta$ and
$\EE[\varepsilon_0]=\eta\delta/2$.
It defines a random distribution inside
$\mathcal E_{\eta,\delta}$, rather than uniform surface measure on that region.
For this setting we report the same as in the first setting in Fig.~\ref{fig: sec4} (middle).
The observed ratio is depicted in Fig.~\ref{fig: sec4} (right), here, we consider $\eta \in \{0.01, 0.1, 0.5\}$.

\begin{figure}[t!]
    \resizebox{\linewidth}{!}{%
    \begin{tabular}{@{}r@{\quad}c@{\quad}c@{}}
    & \multicolumn{2}{c}{\legline[plotblue]{}\;\;Spread spectrum \qquad
    \legline[plotorange]{}\;\;Clustered at $\lambda_2 = \lambda_1 - \delta$ \qquad
    \legline[plotgreen]{}\;\;Clustered at $\lambda_d = 0$} \\
    & \multicolumn{2}{c}{\legline[gray, dashed]{\times}\;\;Local contraction $q^k$ \qquad
    \qquad \legbandunified\;\;$95\%$ bootstrap interval} \\
    \rotatebox{90}{\hspace{80pt} $\EE[\varepsilon_k] / \varepsilon_0$} 
    & \includegraphics[width=0.7\linewidth, clip=true, trim=20pt 20pt 520pt 20pt]{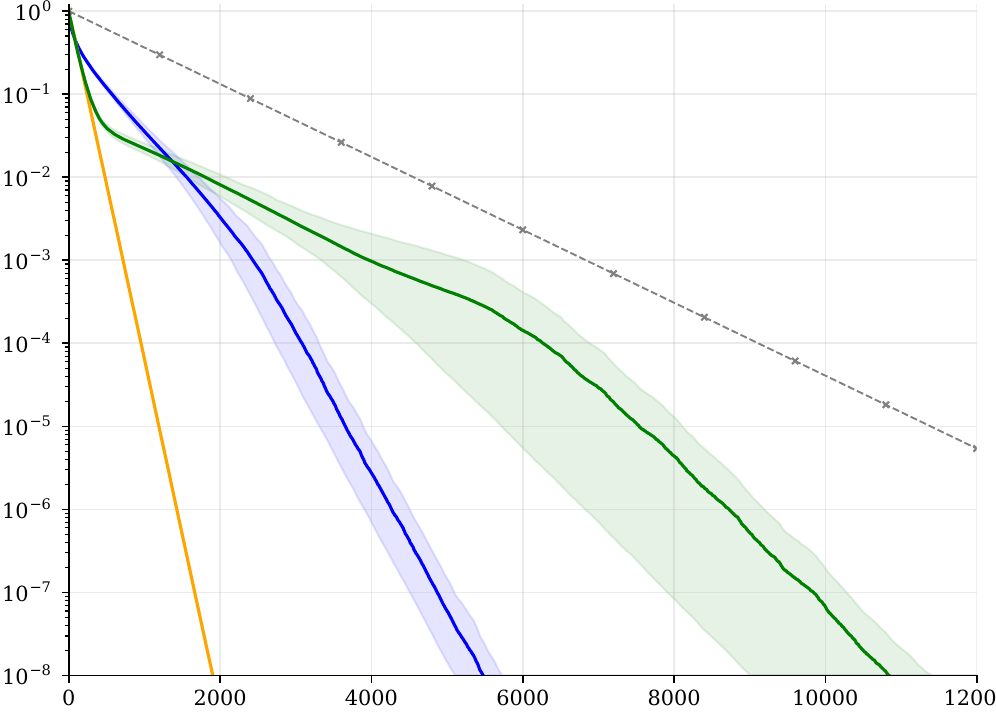} 
    & \includegraphics[width=0.72\linewidth, clip=true, trim=525pt 20pt 0pt 20pt]{spectrum_global_vs_local_gray_eta0.1_bootstrap.pdf} \\[-5pt]
    & \tiny Iteration $k$
    & \tiny Iteration $k$
    \end{tabular}}
    \caption{Comparison of the 
    relative error for three spectra scenarios using uniform random initialization on the unit sphere (\emph{left}), and the random initialization in the local contraction region for $\eta = 0.1$ (\emph{right}).}
    \label{fig: sec4}
\end{figure}

\begin{wraptable}{r}{0.55\textwidth}  
    \resizebox{\linewidth}{!}{%
    \begin{tabular}{p{0.35\linewidth} p{0.28\linewidth} p{0.13\linewidth} p{0.3\linewidth}}
    \toprule 
    Spectral configuration & Observed rate $- \log \hat q$ & Ratio $\tfrac{\log \hat q}{\log q_0}$ & Final mean / initial error \\
    \midrule
    Spread spectrum & $5.905\times10^{-3}$ & $5.84$ & $5.597\times10^{-11}$ \\
    Clustered at $\lambda_2$ & $9.731\times10^{-3}$ & $9.63$ & $5.464\times10^{-22}$ \\
    Clustered at zero & $5.657\times10^{-3}$ & $5.60$ & $7.441\times10^{-8}$ \\
    \bottomrule
    \end{tabular}}
    \caption{Results for \(\eta=0.1\) using $256$ random-local runs as for Fig.~\ref{fig: sec4} (right). }
    \label{tab: sec4}
\end{wraptable}
The main observation in Fig.~\ref{fig: sec4} is that a zero eigenvalue with multiplicity $d-2$ produces the slowest convergence. For the spread spectrum, the rate is intermediate, whereas the configuration with a repeated second-largest eigenvalue yields the fastest observed convergence.
Changing $\eta$ has a comparatively modest effect on the fitted empirical rates. 
Since the theoretical rate $q(\eta,\delta, \Delta)$ increases with $\eta$, the empirical-to-theoretical rate ratios increase with $\eta$, as illustrated in Fig.~\ref{fig: sec4}. 
Table~\ref{tab: sec4} reports the fitted rates $e^{-k\log \hat q}$, their ratios to the theoretical rate, and the final relative mean errors for $\eta=0$, i.e. compared with $q_0 = q(0,\delta,\Delta)$.

\section{Conclusions}
In this paper, we considered a problem of finding the largest singular value of a matrix without involving the adjoint matrix using the random-plane iteration algorithm. Complementing existing results on global almost sure convergence, we established local and global linear convergence rates in the presence of the spectral gap, with the local rate being near-optimal. The global rate depends strongly on the dimension of the underlying problem, and numerical experiments indicate that, in practice, a much faster rate applies and a better global rate could be a subject for future research. We also observe differences between local and global convergence behavior, showing that considering both scenarios is necessary. Interestingly, in the last numerical experiment, we saw how other eigenvalues affect the speed of convergence, which the theoretical analysis does not cover.  

\printbibliography

@article{frankl1990some,
  title={Some geometric applications of the beta distribution},
  author={Frankl, Peter and Maehara, Hiroshi},
  journal={Annals of the Institute of Statistical Mathematics},
  volume={42},
  number={3},
  pages={463--474},
  year={1990},
  publisher={Springer},
  doi={10.1007/BF00049302}
}

@misc{bresch2026genoperatornorm,
      title={Generalization of Zeroth-Order Method for Quotients of Quadratic Functions}, 
      author={Jonas Bresch},
      year={2026},
      archivePrefix={arXiv},
      doi={10.48550/arXiv.2604.26913}, 
}

@article{bresch2024matrixfreestochasticcalculationoperator,
      title={Matrix-free stochastic calculation of operator norms without using adjoints}, 
      author={Jonas Bresch and Dirk A. Lorenz and Felix Schneppe and Maximilian Winkler},
      year={2026},
      pages={1349--1374},
      volume={47},
      issue={3},
      journal={SIAM J. Matrix Anal. Appl.},
      doi={10.1137/25M1772277}
}

@article{bresch2025computingadjointmismatchlinear,
      title={Computing adjoint mismatch of linear maps}, 
      author={Jonas Bresch and Dirk A. Lorenz and Felix Schneppe and Maximilian Winkler},
      year={2026},
      journal={Journal of Computational and Applied Mathematics},
      doi={10.1016/j.cam.2026.117853}, 
      pages={117853},
      volume={488},
      publisher={Elsevier}
}

@article{halko2011randomapproximatematrixdecomposition,
author = {Halko, N. and Martinsson, P. G. and Tropp, J. A.},
title = {Finding structure with randomness: {P}robabilistic algorithms for constructing approximate matrix decompositions},
journal = {SIAM Review},
volume = {53},
number = {2},
pages = {217--288},
year = {2011},
doi = {10.1137/090771806},
}

@article{vonMises1929power,
  author = {von Mises, Richard and Pollaczek-Geiringer, Hilda},
  title  = {Praktische {V}erfahren der {G}lei\-chungsauflösung},
  journal= {Z. Angew. Math. Mech.},
  volume = {9},
  pages  = {152--164},
  year   = {1929},
  doi    = {10.1002/zamm.19290090206},
}

@book{parlett1998symmetriceigenvalueproblem,
author = {Parlett, Beresford N.},
title = {The {S}ymmetric {E}igenvalue {P}roblem},
publisher = {Society for Industrial and Applied Mathematics},
year = {1998},
doi = {10.1137/1.9781611971163},
address = {Philadelphia, PA, USA},
edition   = {},
series = {Classics in Applied Mathematics},
pages = {xxiv + 391}
}

@article{kuczynski1992estimating,
  title={Estimating the largest eigenvalue by the power and {L}anczos algorithms with a random start},
  author={Kuczy{\'n}ski, Jacek and Wo{\'z}niakowski, Henryk},
  journal={SIAM J. Matrix Anal. Appl.},
  volume={13},
  number={4},
  pages={1094--1122},
  year={1992},
  publisher={SIAM},
  doi={10.1137/0613066}
}

@article{li2023stochastic,
  title={Stochastic zeroth-order {R}iemannian derivative estimation and optimization},
  author={Li, Jiaxiang and Balasubramanian, Krishnakumar and Ma, Shiqian},
  journal={Math. Oper. Res.},
  volume={48},
  number={2},
  pages={1183--1211},
  year={2023},
  publisher={INFORMS},
  doi = {10.1287/moor.2022.1302}
}

@article{Polyak1963,
  author = {Polyak, B. T.},
  title = {Gradient Methods for the Minimisation of Functionals},
  journal = {U.S.S.R. Comput. Math. and Math. Phys.},
  volume = {3},
  number = {4},
  pages = {864--878},
  year = {1963},
  doi = {10.1016/0041-5553(63)90382-3}
}

@inproceedings{Karimi2016LinearConvergence,
  author = {Karimi, Hamed and Nutini, Julie and Schmidt, Mark},
  title = {Linear Convergence of Gradient and Proximal-Gradient Methods
           Under the {Polyak--{\L}ojasiewicz} Condition},
  booktitle = {Machine Learning and Knowledge Discovery in Databases},
  series = {Lect. Notes Comput. Sci.},
  volume = {9851},
  pages = {795--811},
  publisher = {Springer},
  year = {2016},
  doi = {10.1007/978-3-319-46128-1_50},
  % url = {https://arxiv.org/abs/1608.04636}
}

@inproceedings{Ji2019ImprovedZeroth-Order,
  author    = {Ji, Kaiyi and Wang, Zhe and Zhou, Yi and Liang, Yingbin},
  title     = {Improved Zeroth-Order Variance Reduced Algorithms and Analysis for Nonconvex Optimization},
  booktitle = {Proceedings of the 36th International Conference on Machine Learning},
  series    = {Proc. Mach. Learn. Res.},
  volume    = {97},
  pages     = {3100--3109},
  publisher = {PMLR},
  year      = {2019},
  url       = {https://proceedings.mlr.press/v97/ji19a.html}
}

@inproceedings{Han2025EfficientOptimizationOrthogonalityConstraint,
  author    = {Han, Andi and Poirion, Pierre-Louis and Takeda, Akiko},
  title     = {Efficient Optimization with Orthogonality Constraint: A Randomized {Riemannian} Submanifold Method},
  booktitle = {Proceedings of the 42nd International Conference on Machine Learning},
  series    = {Proc. Mach. Learn. Res.},
  volume    = {267},
  pages     = {21814--21843},
  publisher = {PMLR},
  year      = {2025},
  url       = {https://proceedings.mlr.press/v267/han25f.html}
}

@article{zhou2025inexact,
  title={Inexact {R}iemannian gradient descent method for nonconvex optimization with strong convergence},
  author={Zhou, Juan and Deng, Kangkang and Wang, Hongxia and Peng, Zheng},
  journal={J. Sci. Comput.},
  volume={103},
  number={3},
  pages={96},
  year={2025},
  publisher={Springer},
  doi={10.1007/s10915-025-02913-1}
}

@article{alimisis2024geodesic,
  title={Geodesic convexity of the symmetric eigenvalue problem and convergence of steepest descent},
  author={Alimisis, Foivos and Vandereycken, Bart},
  journal={J. Optim. Theory Appl.},
  volume={203},
  number={1},
  pages={920--959},
  year={2024},
  publisher={Springer},
  doi = {10.1007/s10957-024-02538-8}
}

@inproceedings{XuKeGao2017,
  author    = {Xu, Zhiqiang and Ke, Yiping and Gao, Xin},
  title     = {A Fast Stochastic {Riemannian} Eigensolver},
  booktitle = {Proceedings of the 33rd Conference on Uncertainty
               in Artificial Intelligence},
  date      = {2017},
  series    = {Proceedings of Machine Learning Research},
  volume    = {R15},
  pages     = {1--10},
  publisher = {PMLR},
  url       = {https://proceedings.mlr.press/r15/xu17a.html},
}

@article{alimisis2021distributed,
  title={Distributed principal component analysis with limited communication},
  author={Alimisis, Foivos and Davies, Peter and Vandereycken, Bart and Alistarh, Dan},
  journal={NeurIPS},
  volume={34},
  pages={2823--2834},
  year={2021},
  url = {https://proceedings.neurips.cc/paper_files/paper/2021/file/1680e9fa7b4dd5d62ece800239bb53bd-Paper.pdf}
}

@inbook{trefethen2005spectrapseudospectra,
title = {{IV.} {T}ransient Effects and Nonnormal Dynamics},
booktitle = {Spectra and {P}seudospectra: {T}he {B}ehavior of {N}onnormal {M}atrices and {O}perators},
author = {Lloyd N. Trefethen and Mark Embree},
publisher = {Princeton University Press},
address = "Princeton, NJ, USA",
pages = {133--192},
doi = {10.1515/9780691213101-006},
year = {2005},
}

@article{li2011concise,
  title={Concise formulas for the area and volume of a hyperspherical cap},
  author={Li, Shengqiao},
  year={2011},
  journal = {Asian J. Math. Stat.}, 
  volume={4},
  number={1},
  pages={66-70},
  doi = {10.3923/ajms.2011.66.70}
}

@article{peterson2017monte,
  title={Monte {C}arlo-based quantitative pinhole {SPECT} reconstruction using a ray-tracing back-projector},
  author={Peterson, Mikael and Gustafsson, Johan and Ljungberg, Michael},
  journal={EJNMMI Phys.},
  volume={4},
  number={1},
  pages={32},
  year={2017},
  publisher={Springer},
  doi = {10.1186/s40658-017-0198-z}
}

@article{zhang2016unmatchedprojback,
author = {Zhang, Hanming and Wang, Linyuan and Li, Lei and Cai, Ailong and Hu, Guoen and Yan, Bin},
title = {Iterative metal artifact reduction for x-ray computed tomography using unmatched projector/backprojector pairs},
journal = {Med. Phys.},
volume = {43},
number = {6 Part 1},
pages = {3019-3033},
doi = {https://doi.org/10.1118/1.4950722},
year = {2016}
}

@Book{buzug2008ct,
  address = {Berlin},
  author = {Buzug, Thorsten M.},
  publisher = {Springer},
  title = {Computed tomography from photon statistics to modern cone-beam {CT}},
  year = {2008},
  address = {Berlin, Heidelberg, Germany},
  doi={10.1007/978-3-540-39408-2}
}

@article{xie2015effective,
    doi = {10.1371/journal.pone.0142184},
    author = {Xie, Lizhe AND Hu, Yining AND Yan, Bin AND Wang, Lin AND Yang, Benqiang AND Liu, Wenyuan AND Zhang, Libo AND Luo, Limin AND Shu, Huazhong AND Chen, Yang},
    journal = {PLOS ONE},
    publisher = {Public Library of Science},
    title = {An Effective {CUDA} Parallelization of Projection in Iterative Tomography Reconstruction},
    year = {2015},
    month = {11},
    volume = {10},
    pages = {1-17},
    number = {11},

}

@article{alimisis2024gradient,
  title={Gradient-type subspace iteration methods for the symmetric eigenvalue problem},
  author={Alimisis, Foivos and Saad, Yousef and Vandereycken, Bart},
  journal={SIAM J. Matrix Anal. Appl.},
  volume={45},
  number={4},
  pages={2360--2386},
  year={2024},
  publisher={SIAM},
  doi={10.1137/23M1590792}
}

@misc{bresch2025stochastic,
  title={Stochastic Zeroth-Order Method for Computing Generalized Rayleigh Quotients},
  author={Bresch, Jonas and Melnyk, Oleh and Schoen, Martin and Steidl, Gabriele},
  archivePrefix={arXiv},
  year={2025},
  doi={10.48550/arXiv.2512.05520}
}

@article{van2016fast,
  title={Fast and flexible X-ray tomography using the ASTRA toolbox},
  author={van Aarle, Wim and Palenstijn, Willem Jan and Cant, Jeroen and Janssens, Eline and Bleichrodt, Folkert and Dabravolski, Andrei and De Beenhouwer, Jan and Batenburg, K. Joost and Sijbers, Jan},
  journal={Optics express},
  volume={24},
  number={22},
  pages={25129--25147},
  year={2016},
  publisher={Optical Society of America},
doi = {10.1364/OE.24.025129},
}

@article{hansen2021GMRESMF,
  title={{GMRES} {M}ethods for Tomographic Reconstruction with an Unmatched Back Projector},
  author={Hansen, Per Christian and Hayami, Ken and Morikuni, Keiichi},
  journal={J. Comput. Appl. Math.},
  year={2021},
  volume={413},
  pages={114352},
doi = {10.1016/j.cam.2022.114352}
}

@book{hansen2021CT,
author = {Hansen, Per Christian and Jørgensen, Jakob and Lionheart, William R.B.},
title = {Computed Tomography: Algorithms, Insight, and Just Enough Theory},
publisher = {Society for Industrial and Applied Mathematics},
year = {2021},
doi = {10.1137/1.9781611976670},
address = {Philadelphia, PA},
}

@article{huber2025convergence,
  title={Convergence of Ray-and Pixel-Driven Discretization Frameworks in the Strong Operator Topology},
  author={Huber, Richard},
  year={2025},
  doi={10.3934/ipi.2026007},
  journal={Inverse Probl. Imaging},
}

@article{delcorso1997estimating,
  author  = {Del Corso, Gianna M.},
  title   = {Estimating an Eigenvector by the Power Method with a Random Start},
  journal = {SIAM Journal on Matrix Analysis and Applications},
  year    = {1997},
  volume  = {18},
  number  = {4},
  pages   = {913--937},
  doi     = {10.1137/S0895479895296689}
}

@article{kazemi2024efficient,
  title={Efficient zeroth-order proximal stochastic method for nonconvex nonsmooth black-box problems},
  author={Kazemi, Ehsan and Wang, Liqiang},
  journal={Machine Learning},
  volume={113},
  number={1},
  pages={97--120},
  year={2024},
  publisher={Springer},
  doi={10.1007/s10994-023-06409-7}
}

\end{document}